\documentclass[11pt,reqno]{amsart}
\usepackage{fullpage}
\usepackage{amsmath}
\usepackage{amsfonts}
\usepackage{amssymb}
\usepackage{graphicx}
\usepackage{mathrsfs}
\usepackage{epsfig}
\usepackage{amsthm}
\usepackage{verbatim}
\usepackage{graphicx}
\usepackage{color}
\usepackage{url}

\usepackage[colorlinks,citecolor=black,linkcolor=black,
            bookmarksopen,
            bookmarksnumbered
           ]{hyperref}

\newcommand{\red}{\color{black}}
\newcommand{\black}{\color{black}}

\newcommand{\beq}{\begin{equation}}
\newcommand{\eeq}{\end{equation}}
\newcommand{\e}{\mathrm{e}}
\newcommand{\dd}{\,\mathrm{d}}

\newcommand{\taul}{\tau^{-1/2}}
\newcommand{\epsi}{\kappa}

\newcommand{\unum}[1]{u^{#1}_\tau}
\newcommand{\un}{\unum{n}}
\newcommand{\unumq}[1]{\overline{u}^{#1}_\tau}
\newcommand{\unq}{\unumq{n}}
\newcommand{\VK}[1]{V^K_{\epsi_{#1}}(s,t_n)}
\newcommand{\WK}[1]{W^K_{\epsi_{#1}}(s,t_n)}
\newcommand{\VqK}[1]{\overline{V}^K_{\epsi_{#1}}(s,t_n)}
\newcommand{\WqK}[1]{\overline{W}^K_{\epsi_{#1}}(s,t_n)}
\newcommand{\PhiKtau}{\Phi^K_\tau}
\newcommand{\Tkappa}{\mathcal T_\kappa}

\newtheorem{theorem}{Theorem}[section]

\newtheorem{lemma}[theorem]{Lemma}
\newtheorem{cor}[theorem]{Corollary}
\newtheorem{proposition}[theorem]{Proposition}
\newtheorem{rem}[theorem]{Remark}

\allowdisplaybreaks[3]

\author{Alexander Ostermann}
\address{Department of Mathematics, University of Innsbruck, Technikerstr.~13, 6020 Innsbruck, Austria (A. Ostermann)}
\email{alexander.ostermann@uibk.ac.at}

\author{Fr\'ed\'eric Rousset}
\address{Laboratoire de Math\'ematiques d'Orsay (UMR 8628), Universit\'e Paris-Sud,  91405 Orsay Cedex, France (F. Rousset)}
\email{frederic.rousset@math.u-psud.fr}

\author{Katharina Schratz}
\address{LJLL (UMR 7598), Sorbonne Universit\'e, UPMC, 4 place Jussieu, 75005, Paris, France (K. Schratz)}
\email{katharina.schratz@ljll.math.upmc.fr}

\begin{document}
\begin{abstract}
In this paper, we propose a new scheme for the integration of the periodic nonlinear Schr\"{o}dinger equation and rigorously prove convergence rates at low regularity. The new integrator has decisive advantages over standard schemes at low regularity. In particular, it is able to handle initial data in $H^s$ for {$0 < s\le 1$}. The key feature of the integrator is its ability to distinguish between low and medium frequencies in the solution and to treat them differently in the discretization. This new approach requires a well-balanced filtering procedure which is carried out in Fourier space. The convergence analysis of the proposed scheme is based on discrete (in time) Bourgain space estimates which we introduce in this paper. A numerical experiment illustrates the superiority of the new integrator over standard schemes for rough initial data.
\end{abstract}

\title[]{Fourier integrator for periodic NLS: low regularity estimates via discrete Bourgain spaces}

\maketitle

%%%%%%%%%%%%%%%%%%%%%%%%%%%%%%
\section{Introduction}
%%%%%%%%%%%%%%%%%%%%%%%%%%%%%%

We consider the cubic periodic Schr\"{o}dinger equation (NLS)
\begin{equation}\label{nlsO}
\begin{aligned}
i \partial_t u &= - \partial_{x}^2 u + \vert u \vert^2 u, \qquad (t,x) \in \mathbb{R} \times \mathbb{T}
\end{aligned}
\end{equation}
which, together with its full space counterpart, has been extensively studied in the literature. In the last decades, Strichartz estimates and Bourgain spaces allowed various authors to establish well-posedness results for dispersive equations in low regularity spaces (see \cite{Bour93a,BGT04,Strichartz,Tao06}). The numerical theory of dispersive PDEs, on the other hand, is still restricted to smooth solutions, in general. In the case of the nonlinear Schr\"{o}dinger equation \eqref{nlsO} this stems from the following two reasons:
\begin{enumerate}
\item[(A)]\label{item:loc}
Standard time stepping techniques, e.g., splitting methods \cite{Lubich08} or exponential integrators \cite{HochOst10}, are based on freezing the free Schr\"{o}dinger flow $S(t) = \e^{i t \partial_x^2}$ during a step of size~$\tau$. Such freezing techniques, related to Taylor series expansion of the linear flow, however, produce derivatives in the local error terms restricting the approximation property to smooth solutions. More precisely, for first-order methods, the expansion of the free flow $S(t+\xi) = S(t)+ \mathcal{O}(\tau \partial_x^2)$, $0<\xi\le\tau$ requires the boundedness of (at least) two additional derivatives, while higher order approximations increase the regularity requirements by two more derivatives for each additional order.
\item[(B)]\label{item:stab}
Standard stability arguments in addition require smooth Sobolev spaces. Indeed, they rely on classical product estimates
$$
\|fg\|_{H^s} \leq C \|f \|_{H^s } \|g\|_{H^s}, \quad s>1/2
$$
to handle the nonlinear terms in the error analysis. %This is due to the fact that in order to control  the nonlinear terms,  the error analysis is performed in $H^s$, $s>1/2$ which is an algebra.
% applies bilinear estimates based on Sobolev embedding
%\begin{equation*}\label{negbiest}
%\begin{aligned}
%& \Vert f g \Vert_{s} \leq c \Vert f\Vert_{s_1} \Vert g \Vert_{s_2} \quad &\text{for  all } &s \leq s_1+s_2-\textstyle\frac{1}{2}\quad &\text{ with } s_1,s_2 \text{ and }-s \neq\textstyle \frac{1}{2}.
%%&\Vert f g \Vert_{s} \leq c \Vert f\Vert_{s_1} \Vert g \Vert_{s_2} \quad &\text{for  all } &s < s_1+s_2-\textstyle\frac{d}{2}\quad &\text{ with } s_1,s_2 \text{ or } -s =\textstyle \frac{d}{2}
%\end{aligned}
%\end{equation*}
This restricts the global error analysis to smooth Sobolev spaces $H^s$ with $s>1/2$ leaving out the important class of $L^2$ spaces.
\end{enumerate}
The standard local error structure introduced by the Schr\"{o}dinger operator, i.e., the loss of two derivatives, together with a standard stability argument thus restricts global first-order convergence to $H^{2+1/2+\varepsilon}$ solutions (for any $\varepsilon>0$). Using a refined global error analysis, by first proving fractional convergence of the scheme in a suitable higher order Sobolev space (which implies a priori the boundedness of the numerical solution in this space \cite{Lubich08}), allows one to obtain stability in $L^2$ for $H^{1/2+\varepsilon}$ solutions. However, due to the standard local error structure $\mathcal{O}(\tau^2 \partial_x^2)$, the first-order convergence rate is nevertheless only retained for $H^2$ solutions. The latter is not only a technical formality. The order reduction in the case of non-smooth solutions is also observed numerically (see, e.g., the examples in \cite{JL00,OS18} and Fig.~\ref{fig} in section~\ref{sec:numerical-ex} below). Only very little is known on how to overcome this problem.

Recently, the first obstacle~(A) could be overcome partly by developing specifically tailored schemes which optimise the structure of the local error approximation. This has been achieved by employing Fourier based techniques that are able to discretize the central oscillations in an efficient and correct way (see \cite{HS16,OS18,OstS19,SWZ20}). The second obstacle~(B), on the other hand, is much harder to circumvent. The control of nonlinear terms in PDEs is an ongoing challenge in (computational) mathematics at large, and unlike in the parabolic setting no pointwise smoothing can be expected for dispersive PDEs. On the continuous level, however, important space time estimates featuring a gain in integrability can be used to extend well-posedness results to lower regularity spaces $H^s$ with  $s<1/2$. On the full space, the Strichartz estimates
\begin{align}\label{str}
\left \Vert \e^{i t \partial_x^2 }u_0 \right\Vert_{L_t^q L_x^r} \leq c_{q,r} \Vert u_0\Vert_2, \qquad  \text{for}\quad
 2 \leq q, r \leq \infty, \quad { 2 \over q } + { 1 \over r}= {1 \over 2}
\end{align}
can be used.  In the periodic setting, though waves do not disperse, one can gain integrability by using Bourgain spaces (we shall give the definition of these spaces  in section \ref{sec:cp}). For \eqref{nlsO} on the torus, the crucial estimate used in the analysis is
\begin{align}\label{B}
\| u \|_{L^4( \mathbb{R} \times \mathbb{T})} \leq C \|u \|_{X^{0, {3 \over 8}}}
\end{align}
which allows for global well-posedness for initial data in $L^2$. We refer for example to \cite{Bour93a,BGT04, Strichartz,Tao06}.

The natural question therefore arises: In how far can we inherit this subtle smoothing property on a discrete level? The critical issue thereby is twofold: the estimates \eqref{str} and~\eqref{B} are not pointwise in time and, moreover, their gain lies in integrability and not regularity. Discrete versions of these estimates are therefore delicate to reproduce. At the same time they are essential to establish numerical stability in the same space where we have stability of the PDE. While discrete Strichartz-type estimates were  successfully employed on the full space $\mathbb{R}^d$ (see, e.g., \cite{IZ09,Ignat11,ORS19}) a global low regularity analysis on bounded domains $\Omega \subset \mathbb{R}^d$ remains an open problem. The step from the full space to the bounded setting is -- as in the continuous setting -- nontrivial due to the loss of dispersion. Strichartz estimates are weaker on bounded domains as the solution can not ``disperse'' to infinity in space. Nevertheless bounded domains are computationally very interesting as spatial discretizations of nonlinear PDEs are in general subjected to truncated domains. %However, up to our knowledge nothing is known on  discrete Bourgain type estimates in numerical literature so far.

In this work we introduce discrete Bourgain spaces for the periodic Schr\"{o}dinger equation \eqref{nlsO}. This will allow us to break standard stability restrictions on the torus. In particular, we establish a discrete version of \eqref{B} permitting $L^2$ error estimates in $H^s$ also for $s\leq 1/2$. For the discretization of~\eqref{nlsO} we propose a new  twice-filtered Fourier based technique that correctly discretizes the central oscillations of the problem. The novel discretisation approach can be applied to a larger class of dispersive equations.  For simplicity we restrict our attention to the cubic Schr\"{o}dinger equation. The stability analysis (based on discrete Bourgain space estimates) we develop can be extended to various numerical schemes, e.g., splitting methods. The benefit of the here introduced twice-filtered Fourier based approach it that it optimises the local error structure which allows for convergence under lower regularity than standard discretizations.
%In the theorem, $\un$ denotes the numerical solution, approximating the exact solution $u(t)$ at time $t = n \tau$ with step size $\tau$.
The precise form of the scheme is given in~\eqref{scheme} below. In particular, it involves a frequency localization
 through a filter $\Pi_{K}$ (see \eqref{PiKdef}) which projects on frequencies $|k| \leq K$ which will allow us to optimize the total (time and frequency) discretization error.   Indeed, \black with the help of discrete Bourgain type estimates we can prove the following global error estimate  for the new scheme.
{
\begin{theorem}
\label{maintheo}
For every $T >0$ and $u_{0} \in H^{s_{0}}$,  $0 \leq s_{0} \leq 1$, let us denote by $u \in \mathcal{C}([0, T], H^{s_{0}})$ the exact solution of \eqref{nlsO} with initial datum $u_{0}$ and by $\un$ the sequence defined by the scheme~\eqref{scheme} below. Then, 
 we have  the following error estimates:
\begin{itemize}
\item[(i)] For $K= \tau^{- {1 \over 2}}$ and $0< s_0\le \frac14$, 
there exists $\tau_{0}>0$ and $C_{T}>0$ such that for every step size $\tau \in (0, \tau_{0}]$
\beq
\label{final1} \|\un- u(t_{n})\|_{L^2} \leq C_{T} \tau^{s_{0}\over 2}, \quad 0 \leq n\tau \leq T;
\eeq 
\item[(ii)] For  $\frac14 < s_0 \le \frac12$, and any $\varepsilon>0$ such that
 $ {1\over 4} + \varepsilon <s_{0}$,  with the choice $ K= \tau^{-\frac{ s_{0} + \frac1{8} - {\varepsilon \over 2} }{s_{0}+ \frac12}}$, 
 there exists $\tau_{0}>0$ and $C_{T}>0$ such that for every step size $\tau \in (0, \tau_{0}]$, we have
\beq
\label{final2}
\|\un- u(t_{n})\|_{L^2} \leq C_{T} \tau^{  s_{0}\left(  1 -  { 1 \over 2 s_{0}+ 1} ( \frac{3}{4}+ \varepsilon)\right)}, \quad 0 \leq n\tau\le T;
\eeq
\item[(iii)] For $\frac12 <s_0 \leq 1$,  and $\varepsilon
 \in (0, {1 \over 4})$, 
  with the choice $ K=  \tau^{ -1 + { 1 \over s_{0}} ( { 1 \over 8} + {\varepsilon \over 2})   }$, 
  there exists $\tau_{0}>0$ and $C_{T}>0$ such that for every step size $\tau \in (0, \tau_{0}]$, we have   
\beq
\label{final3}  \|\un- u(t_{n})\|_{L^2} \leq C_{T}   \tau^{ s_{0}-  ( { 1 \over 8} +{\varepsilon \over 2})}, \quad 0 \leq n\tau\le T.
\eeq
\end{itemize}
\end{theorem} }
%We have used the notation $x_+$ for any $y>x$.

In case (i), the error estimate we obtain for our new Fourier based discretization is not better than the one we would expect for standard schemes (based on classical Taylor series expansion techniques). The interesting feature, however, is that the analysis we develop is able to provide an error estimate even for data at this low level of regularity.  So far error estimates (even with arbitrary low order of convergence) were restricted to solutions (at least) in  \red $H^{s}$, $s>1/2$\black. We note that our analysis can be employed for a large class of schemes, e.g., splitting methods or exponential integrators. 

In the cases (ii) and (iii), we observe that we get a better estimate than $\tau^{s_{0}\over 2}$ which is the one we would expect for standard numerical schemes with a loss of two derivatives in the local error (cf. (A)). Observe, for example, that for $s_{0}=1$, we get an error estimate of order $\tau^{7 \over 8}$ which is much better than the standard $\tau^{1 \over 2}.$   Note that it is even (slightly) better than the convergence order $\tau^{ 5 \over 6}$ which we obtained with the help of discrete Strichartz estimates on the full space in \cite{ORS19} though dispersive effects are much weaker in the periodic case. This comes from our improved Fourier based discretization with the use of the two different filters $\Pi_{\tau^{-1/2}}$ and $\Pi_{K}$. The favourable error behaviour is numerically underlined in Fig.~\ref{fig} (see section~\ref{sec:numerical-ex} below).

\red
It seems also possible to extend the analysis developped in this paper to higher dimensions and more general nonlinearities.
A large part of the framework that we introduce can be readily extended, 
the central task would lie in establishing  
 the corresponding discrete counter part of the continuous Bourgain estimates given in  \cite{Bour93a}
in the various cases as done here  in Lemma \ref{prodd}. 

\black

The main idea in our discretization is the following. Instead of attacking directly ~\eqref{nlsO}, we discretize the projected equation
\begin{multline}\label{nls}
i \partial_t u^K = - \partial_{x}^2 u^K
+ 2 \Pi_K \left( \Pi_{\taul} u^K \,\Pi_{\taul} \overline u^K\, \Pi_{K^+} u^K \right)
+ \Pi_K \left( \Pi_{\taul} u^K \,\Pi_{\taul} u^K\, \Pi_{K}  \overline  u^K \right),
\end{multline}
where the projection operator $\Pi_L$ for $L>0$ is defined by the Fourier multiplier
\beq
\label{PiKdef}
\Pi_L =  \chi^2\left({ -i \partial_x \over L}\right) = \overline\Pi_L,
\eeq
and  where $\Pi_{K^+}$ projects on the intermediate frequencies $ \taul \le\vert  k\vert \leq K$, i.e.,
\begin{equation}\label{Kp}
\Pi_{K^+} = \Pi_K -\Pi_{\taul}.
\end{equation}
Here $\chi$ is a smooth nonnegative  even  function which is one on $[-1,1]$ and supported in $[-2,2]$. The number $K\geq 1$ is considered as a parameter that will later depend on the step size $\tau$.  Note that the projection operator $\Pi_K$  in Fourier space reads
$$
\widehat{\Pi_K\phi}_\ell = \widehat{\phi}_\ell\, \chi^2\left(\frac{\ell}{K}\right), \quad \ell \in \mathbb{Z}.
$$
 \red  Splitting methods with  numerical filters have been successfully introduced in  \cite{Bao} for nonlinear Schr\"odinger equations  in the semiclassical regime  with attractive interaction to suppress numerically the modulation instability. \black \red Here, the relation between $K$ and $\tau$ can be seen as a CFL-type condition linking  the time discretisation parameter to the highest frequency in the system. We optimise this relation in such a way that the optimal rate of convergence is achieved for a given regularity; see Theorem \ref{maintheo}. \black 
 
The reason why we base our discretization on \eqref{nls} is twofold. First, we consider
\beq
\label{eq:frectrunc}
i \partial_{t} v^K = - \partial_x^2  v^K + \Pi_{K}(|\Pi_{K}v^K|^2 \Pi_{K}v^K), \quad  v^K(0)= \Pi_{K} u_{0}
\eeq
as an intermediate problem the single-filtered equation where all high frequencies are truncated. The difference between solutions of \eqref{nlsO} and \eqref{eq:frectrunc} is estimated in Corollary~\ref{corNLSK} and easy to control. Second, we refine the truncated model \eqref{eq:frectrunc} by considering a second projection $\Pi_{\taul}$ to low frequencies. Roughly speaking, each function with frequencies below $K$ is then decomposed into two parts: low frequencies for which $|k|\le \taul$ and the remaining intermediate frequencies. Since the original problem is cubic, these two projections lead to six terms in total. For our discretization, we only consider those terms in which two of the factors are of low frequencies. This motivates us to consider the twice-filtered equation \eqref{nls} as an approximation to equation \eqref{nlsO}.

The discretization of the twice-filtered Schr\"{o}dinger equation \eqref{nls} is carried out in a way such that the terms with intermediate frequencies $ \taul < \vert k \vert \leq K$ are treated  {\em exactly} while the lower order terms with frequencies $|k| \leq \taul$ are approximated in a suitable manner. This approach allows for a {\em low regularity approximation} of solutions of \eqref{nlsO}. Motivated by our previous work \cite{KOS19}, we thus propose the following numerical scheme:

\begin{equation}\label{scheme}
\begin{aligned}
\unum{n+1}  & = : \PhiKtau(\un) \\
& = \e^{i \tau \partial_x^2} \un - 2i \Pi_K \e^{i \tau \partial_x^2}
\mathcal{J}_1^\tau\left( \Pi_{\taul} \unq, \Pi_{K^+} \un, \Pi_{\taul} \un\right)\\
& \qquad\qquad\ \,
- i  \Pi_K \e^{i \tau \partial_x^2}\mathcal{J}_2^\tau\left(\Pi_{K} \unq,\Pi_{\taul} \un, \Pi_{\taul} \un\right),\\
\unum{0} & = \Pi_K  u(0),
\end{aligned}
\end{equation}
where
\begin{align}%\label{J}
\mathcal{J}^\tau_1(v_1,v_2,v_3) &
=  \frac{i}{2} \e^{-i \tau\partial_x^2}\left[ \left(\e^{i \tau \partial_x^2} \partial_x^{-1 }v_2 \right)\e^{i \tau \partial_x^2} \partial_x^{-1} \red(\black v_1v_3 \red)\black\right] \nonumber \\
&\qquad - \frac{i}{2} \left( \partial_x^{-1} v_2\right) \partial_x^{-1} \red(\black  v_1v_3  \red)\black + \tau \widehat{(v_2)}_0 v_1v_3 +\tau (v_2-\widehat{(v_2)}_0) \widehat{(v_1v_3)}_0, \label{J1}\\
\mathcal{J}^\tau_2(v_1,v_2,v_3) &
= \frac{i}{2} \e^{-i \tau \partial_x^2}\partial_x^{-1} \Big[\left(\e^{-i\tau \partial_x^2} \partial_x^{-1} v_1 \right)\left(\e^{i \tau \partial_x^2} v_2v_3\right)\Big] \nonumber\\
&\qquad - \frac{i}{2} \partial_x^{-1} \left(v_2v_3 \partial_x^{-1} v_1\right) + \tau \widehat{\left(v_1 v_2v_3\right)}_0 + \tau \widehat{(v_1)}_0 \big(v_2v_3 -  \widehat{(v_2v_3)}_0 \big)\label{J2}.
\end{align}
Here, we define for any function $f \in L^2(\mathbb{T})$ the operator $\partial_x^{-1}$ by $\partial_x^{-1} f(x) = \sum_{\substack{k \neq 0}} {(i k)^{-1}}{\hat{f}_k } \e^{i k x}$. Note that $\un$ in \eqref{scheme} is considered as an approximation to the exact solution of the nonlinear Schr\"{o}dinger equation \eqref{nlsO} at time $t_n=n\tau$.

%%%%%%%%%%%%%%%%%%%%%%%%%%%%%%%%%%
\subsection*{Outline of the paper}
%%%%%%%%%%%%%%%%%%%%%%%%%%%%%%%%%%
The paper is organized as follows. In section \ref{sec:cp}, we recall the main steps of the analysis of the Cauchy problem for \eqref{nlsO} and we use them to estimate the difference between the exact solution of~\eqref{nlsO} and the solution of the projected equation \eqref{nls}.  In particular, we prove that
\beq\label{spacediscintro}
\sup_{[0, T]} \| u - u^K \|_{L^2} \leq C_{T}\left( { 1 \over K^{s_{0}}} + \tau^{s_{0}}\right),
\eeq
see Corollary~\ref{corNLSK} and Proposition~\ref{propNLSK2}.

In section \ref{sec:discbourg}, we introduce a notion of  discrete Bourgain spaces for sequences
   $(u_{n})_{n} \in (L^2(\mathbb{T}))^\mathbb{N}$ and prove their main properties.
    The crucial property for the error analysis is the  following $L^4$ estimate,
    $$ \| \Pi_{K} u_{n} \|_{l^4_{\tau}L^4} \leq (K \tau^{1 \over 2})^{1 \over 2} \|u_{n}\|_{X^{0, {3\over 8}}_{\tau}}$$
 which holds uniformly for    $K \geq \tau^{- {1 \over 2}}$ and  $0<\tau \leq 1$. This estimate is proven in section \ref{sectiontechnic}.
 The $l^p_{\tau}$ norm for vector valued sequences is defined in \eqref{normelptau}.
  From this property, we see that the choice $K= \tau^{- { 1 \over 2}}$ allows one to get an estimate without loss
  similar to the continuous case \eqref{B}. Nevertheless, such a choice of $K$ yields a rather bad space discretization
   error \eqref{spacediscintro}. We shall thus optimize $K$ by taking it of the form
   $ K= \tau^{- { \alpha  \over 2}}$ for $\alpha \in [1, 2]$ to get the best possible total error.

In section \ref{section4}, we establish embedding estimates between discrete and continuous Bourgain spaces.

In section \ref{section5}, we analyze the local error of our scheme and in section \ref{section6}, we provide global error estimates. Finally in section \ref{section7}, we prove the main error estimate of Theorem \ref{maintheo}.

We conclude in section \ref{sec:numerical-ex} with numerical experiments underlying the favorable error behavior of the new scheme for rough data.

%%%%%%%%%%%%%%%%%%%%%%%
\subsection*{Notations}
%%%%%%%%%%%%%%%%%%%%%%%

We close this section with some notation that will be used throughout the paper.
For two expressions $a$ and $b$, we write $a \lesssim b$ whenever $a \leq C b$ holds with some constant $C>0$, uniformly in $\tau \in (0, 1]$ and $K \geq 1$.  We further write $a\sim b$ if $b\lesssim a \lesssim b$. When we want to emphasize that $C$ depends on an additional parameter $\gamma$, we write $a \lesssim_{\gamma} b$.
We shall also use  the notation $x_+$ for any $y>x$.

Further, we denote $\langle \,\cdot\, \rangle = ( 1 + | \cdot |^2)^{1 \over 2}$.

%%%%%%%%%%%%%%%%%%%%%%%%%%%%%%%%%%%%%%%%%%%%%%%%%%%%%%%%%%%%%%%%%%%%%%%%%%%%%%%%%%%%%%%%%%%%%
\section{Cauchy problem for \eqref{nlsO}}\label{sec:cp}
%%%%%%%%%%%%%%%%%%%%%%%%%%%%%%%%%%%%%%%%%%%%%%%%%%%%%%%%%%%%%%%%%%%%%%%%%%%%%%%%%%%%%%%%%%%%%

Let us recall the definition of Bourgain spaces. A tempered distribution $u(t,x)$ on $\mathbb{R} \times \mathbb{T}$ belongs to the Bourgain space $X^{s, b}$ if its following norm is finite
\begin{equation}
\label{defnormeBourgain}
\|u\|_{X^{s, b}}=\left(\int_{\mathbb{R}}\sum_{k \in \mathbb{Z}}\left(1+ |k|\right)^{2s}\left(1+| \sigma+k^2|\right)^{2b}|\widetilde{u}\left(\sigma, k\right)|^2 \dd\sigma\right)^{\frac{1}{2}},
\end{equation}
where $\widetilde{u}$ is the space-time Fourier transform of $u$:
$$
\widetilde{u}(\sigma, k)= \int_{\mathbb{R}\times \mathbb{T}} \e^{-i \sigma t - i k x} u(t,x) \, \dd t \dd x.
$$
We shall also use a localized version of this space, $u \in X^{s,b}(I)$ where $I \subset \mathbb{R}$ is an open  interval
 if $\|u\|_{X^{s,b}(I)}< \infty$, where
$$
\|u\|_{X^{s,b}(I)} = \inf\{\|\overline{u} \|_{X^{s,b}}, \, \overline{u}|_{I} = u  \}.
$$
When $I= (0, T)$ we will often simply use the notation $X^{s,b}(T)$.

We shall recall now well-known properties of these spaces. For details, we refer for example to \cite{Bour93a}, and the books \cite{Linares-Ponce}, \cite{Tao06}.
\begin{lemma}
\label{lemmabourgainfacile}
For $\eta\in \mathcal{C}^\infty_{c}(\mathbb{R})$, we have that
\begin{eqnarray}
\label{bourg1an}
&  \| \eta (t) \e^{i t \partial_x^2 } f \|_{X^{s,b}} \lesssim_{\eta, b} \|f\|_{H^s}, \quad s \in \mathbb{R},\, b\in \mathbb R,\, f \in H^s(\mathbb{T}), \\
  \label{bourg2an}
&   \| \eta (t) u  \|_{X^{s,b}} \lesssim_{\eta, b} \|u\|_{X^{s,b}}, \quad s \in \mathbb{R}, \, b\in\mathbb R, \\
   \label{bourg3an}
&  \| \eta({t \over T}) u \|_{X^{s,b'}} \lesssim_{\eta, b, b'} T^{b-b'} \|u\|_{X^{s,b}}, \quad s \in \mathbb{R},  -{1 \over 2} <b' \leq b <{ 1 \over 2}, \, 0< T \leq 1, \\
& \left\| \eta(t) \int_{-\infty}^t \e^{i (t-s) \Delta} F(s) \dd s \right\|_{X^{s,b}} \lesssim_{\eta, b} \|F\|_{X^{s, b-1}}, \quad s \in \mathbb{R}, \,
 b>{1 \over 2}, \\
 &  \label{bourg4an}  \|u\|_{L^\infty(\mathbb{R}, H^s)} \lesssim_{b} \|u\|_{X^{s, b}}, \quad b>1/2, \, s \in \mathbb{R}.
\end{eqnarray}
\end{lemma}
We actually have  the continuous embedding $X^{s, b} \subset \mathcal{C}(\mathbb{R}, H^s)$ for $b>1/2$. Note that we shall discuss below an  extension of the definition of the Bourgain spaces and of this lemma to a discrete setting suitable for the analysis of numerical schemes and give the proofs in this discrete  setting.

The crucial estimate for the analysis of the cubic NLS on the torus $\mathbb{T}$ is the following:

\begin{lemma}
\label{lemmabourgaindur}
There exists a constant $C>0$ such that for every $u \in X^{0, {3\over 8}}$, we have the estimate
$$
\| u \|_{L^4( \mathbb{R} \times \mathbb{T})} \leq C \|u \|_{X^{0, {3 \over 8}}}.
$$
\end{lemma}
Again, we refer  to \cite{Tao06} Proposition 2.13  for its proof. Note that, by duality, we also obtain that
$$
\| u \|_{X^{0, {-{3 \over 8}}}} \lesssim \| u \|_{L^{4\over 3}( \mathbb{R} \times \mathbb{T})}.
$$
By combining the two estimates with H\"{o}lder, this further implies that
\begin{equation}
\label{prodbourgain}
\| u v w  \|_{X^{0, {-{3 \over 8}}}} \lesssim \|u\|_{X^{0, {3 \over 8}}}  \|v\|_{X^{0, {3 \over 8}}}   \|w\|_{X^{0, {3 \over 8}}}.
\end{equation}

For \eqref{nlsO}, we have the following global well-posedness result.

\begin{theorem}
\label{theoNLS}
For every $T>0$ and $u_{0}\in L^2$, there exists a unique solution $u$ of \eqref{nlsO} such that $u\in \mathcal{C}([0, T], L^2) \cap X^{0, {b}}(T)$ for  any  $b \in (1/2, 5/8)$. Moreover, if $u_{0} \in H^{s_{0}}$, $s_{0}> 0$, then $ u\in \mathcal{C}([0, T], H^{s_{0}}) \cap X^{s_{0}, b}(T)$.
\end{theorem}

\begin{proof}
Let us recall the main steps of the proof. The existence is proven by a fixed point argument on the following truncated problem:
$$
v\mapsto F(v)
$$
such that
\begin{equation}\label{fixed}
F(v)(t)=   \eta(t)\e^{i t \partial_{x}^2} u_{0} -  i \eta(t)\int_{0}^t \e^{i(t- s)\partial_{x}^2 }
\left( \eta\left({s \over \delta}\right) |v(s)|^2 v(s)\right)\dd s,
\end{equation}
where $\eta \in [0, 1]$ is a  smooth compactly supported function which is equal to $1$ on $[-1, 1]$ and supported in $[-2, 2]$. For $|t| \leq \delta \leq 1/2 $, a fixed point of the above equation gives a solution of the original Cauchy problem, denoted by $u$.

Thanks to Lemma \ref{lemmabourgainfacile}, there exists $C>0$ which does not depend on $u_{0}$ such that
$$
\|\eta(t)\e^{it \partial_{x}^2} u_{0}\|_{X^{0, {b}}} \leq C \|u_{0}\|_{L^2}.
$$
Moreover, by using Lemma \ref{lemmabourgainfacile} and \eqref{prodbourgain}, we can estimate the Duhamel term by
\begin{multline*}
\left\| \eta(t)\int_{0}^t \e^{i(t- s)\Delta }
\left( \eta\left({s \over \delta}\right) |v(s)|^2 v(s)\right)\!\dd s \right\|_{X^{0, b}}
\leq C  \left\| \eta\left({s \over \delta}\right) |v(s)|^2 v(s)\right\|_{X^{0, b-1 } } \\
\leq C \delta^{ \epsilon_{0}}  \left\|  |v(s)|^2 v(s)\right\|_{X^{0, - \frac38}} \leq C  \delta^{ \epsilon_{0}} \|v\|_{X^{0, \frac38}}^3 \leq  C  \delta^{ \epsilon_{0}} \|v\|_{X^{0, b}}^3,
\end{multline*}
where $C>0$ is again a generic constant and $\epsilon_{0}= 5/8 - b >0$ by the choice of $b$. Therefore, we have obtained that
$$
\|F(v) \|_{X^{0, b}}  \leq C \|u_{0}\|_{L^2} + C \delta^{ \epsilon_{0}}  \|v\|_{X^{0,b}}^3.
$$
In a similar way, we obtain that if $v_{1}$ and $v_{2}$ are such that $\|v_{1}\|_{X^{0, b}} \leq R, \, \|v_{2} \|_{X^{0, b}}\leq R$, then
$$
\|F(v_{1}) - F(v_{2}) \|_{X^{0, b}} \leq 4 C \delta^{\epsilon_{0}} R^2 \|v_{1} - v_{2} \|_{X^{0, b}}.
$$
Consequently, by taking $R= 2 C \|u_{0}\|_{L^2}$, we get that there exists $\delta>0$ sufficiently small that depends only on $\|u_{0}\|_{L^2}$ such that $F$ is a contraction on the closed ball $B(0, R)$ of $X^{0, b}.$ This proves the existence of a fixed point $v$ for $F$ and hence the existence of a solution $u$  of \eqref{nlsO} on $[0, \delta]$. By using Lemma \ref{lemmabourgainfacile}, we actually get that $u\in \mathcal{C}([0, \delta], L^2)$.  Since  for $s \geq 0$,
$$
\|F(v)\|_{X^{s, b}}  \leq C \|u_{0}\|_{H^s} + C \delta^{ \epsilon_{0}}  \|v\|_{X^{0, b}}^2 \|v\|_{X^{s, b}},
$$
we also get that if $u_{0}$ is in $H^s$ then $u \in X^{s,b}( [0, \delta])$. Since the $L^2$ norm is conserved for \eqref{nlsO}, we can  reiterate the construction on $[\delta, 2 \delta], \ldots$ to get a global solution. Moreover, since $\delta$ depends only on the $L^2$ norm of $u_{0}$, we get that if $u_{0}$ is in $H^s$, $s \geq 0$ then $u \in X^{s, b}(T)$ and thus $u\in \mathcal{C}([0, T], H^s)$ for every $T$.
\end{proof}

Let us now consider $v^K$ that solves the frequency truncated equation
\beq
\label{NLSK}
i \partial_{t} v^K = - \partial_x^2  v^K + \Pi_{K}(|\Pi_{K}v^K|^2 \Pi_{K}v^K), \quad  v^K(0)= \Pi_{K} u_{0}.
\eeq

As in Theorem \ref{theoNLS}, we can easily get:

\begin{proposition}
\label{propNLSK}
For  $u_{0} \in H^{s_{0}}$,  $s_{0} \geq 0$, and $K \geq 1$, there exists a unique solution $v^K$ of \eqref{NLSK} such that $v^K \in X^{s_{0}, b}(T)$ for $b\in(1/2,5/8)$ and every $T > 0$. Moreover, for every $T>0$, there exists $M_{T}$ such that for every $K \geq 1$, we have the estimate
$$
\| v^K \|_{X^{s_{0}, b}(T)} \leq M_{T}.
$$
\end{proposition}

We shall not detail the proof of this proposition that follows exactly the lines of the proof of Theorem \ref{theoNLS}.

\begin{rem}
\label{remlocal}
Since $\Pi_{2K} \Pi_{K}= \Pi_{K}$, we have that $\Pi_{2K} v^{K}$ solves the same equation \eqref{NLSK} with the same initial data. Hence, by uniqueness, we have that
$$
\Pi_{2K} v^K (t) = v^K(t) \quad \text{for all \ $t \in [0,T]$.}
$$
\end{rem}
We can also easily get the following corollary.
\begin{cor}
\label{corNLSK}
For  $u_{0} \in H^{s_{0}}$, $s_{0} \geq 0$  and every $T>0$, there exists $C_{T}>0$ such that for every %$K =\tau^{- {\alpha \over 2} } \geq 1$, $ \alpha \geq 1$
$K\ge 1$ we have the estimate
$$
\|u -v^K \|_{X^{0,b }(T)} \leq  C_{T} K^{-s_0},%\tau^{ s_{0} \alpha \over 2}
$$
where $b$ is as in Theorem \ref{theoNLS}.
\end{cor}

\begin{proof}
For appropriately chosen $\delta >0$ and $\eta$ as in \eqref{fixed}, we observe that on $[0, \delta]$,  $v^K$ is the restriction  of $V^K \in X^{s_{0}, {3 \over 8}}(\mathbb{R})$ that solves
$$
V^K(t)=   \eta(t)\e^{i t \partial_{x}^2} v^K(0) -  i \eta(t)\int_{0}^t \e^{i(t- s)\partial_{x}^2 }
\left( \eta\left({s \over \delta}\right) \Pi_{K}\Bigl( |\Pi_{K}V^K(s)|^2  \Pi_{K}V^K(s)\Bigl)\right)\dd s.
$$
Consequently, by denoting $U \in X^{s_{0}, {3 \over 8}}(\mathbb{R})$ the fixed point of  $F$ such that on $[0, \delta]$, $U= u$,   we obtain that
\begin{align*}
U (t) - V^K(t) &=  \eta(t) \e^{it \partial_x^2}  (1-\Pi_{K}) u_{0}\\
&\quad - i  \eta(t) \int_{0}^t \e^{i (t-s ) \partial_x^2 } \left( \eta\left({s \over \delta}\right)  \Pi_{K}\Bigl( |U(s)|^2 U(s) - |\Pi_{K}U(s)|^2 \Pi_{K} U(s) \Bigr) \right) \dd s  \\
&\quad - i \eta(t) \int_{0}^t \e^{i (t-s ) \partial_x^2 }  \left( \eta\left({s \over \delta}\right) \Pi_{K} \Bigl( |\Pi_{K}U(s)|^2 \Pi_{K} U(s) - |\Pi_{K}V^K(s)|^2 \Pi_{K} V^K(s) \Bigr)\right) \dd s \\
&\quad -i \eta(t) \int_{0}^t \e^{i (t-s ) \partial_x^2 }  \left( \eta\left({s \over \delta}\right) (1- \Pi_{K}) \Bigl( |U(s)|^2 U(s) \Bigr)\right)\dd s.
\end{align*}
Now, let us fix $M_{T}$ independent of $K\geq 1$ such that
$$ \| V^K \|_{X^{s_{0}, b}}  + \| U \|_{X^{s_{0}, b}}  \leq M_{T}.$$
Note that for every $f$, we have the estimate
$$
\| f - \Pi_{K} f \|_{X^{0,b}} \lesssim  {1 \over K^{s_{0}}} \| f \|_{X^{s_{0}, b}}.
$$
By employing the same estimates as before, we thus obtain that
\begin{align*}  \|U-V^K \|_{X^{0, b}}
&  \lesssim  { 1 \over K^{s_{0}}} \|u_{0}\|_{H^{s_{0}}} + \delta^{\epsilon_{0}} \| U - \Pi_{K} U \|_{X^{0,b}} \|U\|_{X^{0, b}}^2
   + \delta^{\epsilon_{0}} \|U -V^K \|_{X^{0, b}}  ( \|U\|_{X^{0, b}}^2 +  \|V^K\|_{X^{0, b}}^2) \\
&  \quad + \delta^{\epsilon_{0}} { 1 \over K^{s_{0}}}  \||U|^2 U \|_{X^{s_{0},  b-1}} \\
& \lesssim   { 1 \over K^{s_{0}}}( \|u_{0}\|_{H^{s_{0}}} + M_{T}^3)  + \delta^{\epsilon_{0}} M_{T}^2  \|U-V^K \|_{X^{0,b}}.
\end{align*}
For $\delta$ sufficiently small, this yields the desired estimate. We can then iterate in order to get the estimate on $[0, T].$
\end{proof}

Instead of performing directly a time discretization of equation \eqref{NLSK}, it will be convenient for the analysis to study a slightly modified equation. Let $u^K$ be the solution of
\begin{equation*}
i \partial_t u^K = - \partial_{x}^2 u^K
+ 2 \Pi_K \left( \Pi_{\taul} u^K \,\Pi_{\taul} \overline u^K\, \Pi_{K^+} u^K \right)
+ \Pi_K \left( \Pi_{\taul} u^K \,\Pi_{\taul} u^K\, \Pi_{K}  \overline  u^K \right)
\end{equation*}
(cf.~\eqref{nls}), again with the initial data $u^K(0) = \Pi_{K} u_{0}$. Note that the difference between this truncated equation and \eqref{NLSK} is in the trilinear terms, where we can always project at least two factors on frequencies less than $ \tau^{- 1/2}$. %{1 \over 2}}$.
Further note that $u^K$ depends also on $\tau$ though we do not explicitly mention it in order to keep reasonable notation. However, we will henceforth link $K$ and $\tau$ by the relation
$$
K = \tau^{-\frac{\alpha}2},\quad \alpha \ge 1,
$$
with the (optimal) value of $\alpha$ still to be determined.

Again, we have existence and uniqueness of the solution.

\begin{proposition}
\label{propNLSK2}
For  $u_{0} \in H^{s_{0}}$, $s_{0} \geq 0$  and $K \geq 1$, there exists a unique solution $u^K$ of \eqref{nls} such that $u^K \in X^{s_{0}, b}(T)$ (with $b$ as in Theorem \ref{theoNLS}) for every $T > 0$. Moreover, for every $T>0$, there exists $M_{T}$ such that for every $K \geq 1$, we have the estimate
$$
\| u^K \|_{X^{s_{0}, {b}}(T)} \leq M_{T}.
$$
Further, it holds that uniformly for $K \geq 1$,
$$
\| u^K - v^K \|_{X^{0, b}(T)} \leq  C_{T}  \tau^{s_{0}}.
$$
\end{proposition}

Observe that by combining the last estimate with the estimate of Corollary \ref{corNLSK}, we actually get that
\beq
\label{u-uK}
\|  u - u^K\|_{X^{0, b}(T)} \leq  C_{T}  \tau^{s_{0} \alpha \over 2}
\eeq
for $K = \tau^{ - { \alpha \over 2} }$ and  $\alpha$ such that $1 \leq \alpha \leq 2$.

\begin{proof}
The proof of the first part follows again  the lines of the  proof of Theorem \ref{theoNLS}. Let us explain how to prove the error estimate. Let us denote by $G_{\tau, K}(u^{K})$ the nonlinear term on the right-hand side of \eqref{nls}. We first observe that we can write
$$
\Pi_{K}( |\Pi_{K}v^K|^2 \Pi_{K} v^{K})=  G_{\tau, K}(v^{K}) + R_K(v^K),
$$
where the remainder is a sum of terms of the form
$$
R_K(v^K)= \sum_{(i_{1}, i_{2}, i_{3})} \Pi_{K}( Q_{i_{1}} v^{K} Q_{i_{2}} \overline{v^K} Q_{i_{3}} v^{K})
$$
and where in the sum $Q_{i}$ can be $ \Pi_{\tau^{-1/2}}$ or $ (1 - \Pi_{\tau^{-1/2}})\Pi_{K}$ and at least two different $Q_{i}$ are $ (1 - \Pi_{\tau^{-1/2}})\Pi_{K}$.
Let us denote again by $U^K$ and $V^K$ the fixed points of the extended Duhamel formulation such that uniformly for $K \geq 1$, we have
$$
\| V^K \|_{X^{s_{0}, {3 \over 8}}}  + \| U^K \|_{X^{s_{0}, {3 \over 8}}}  \leq M_{T}.
$$
Then we get that
\begin{multline*}
U^K (t) - V^K(t) =   - i  \eta(t) \int_{0}^t \e^{i (t-s ) \Delta } \left( \eta\left({s \over \delta}\right) \Bigl(G_{\tau, K}(U^K(s)) - G_{\tau, K}(V^K(s))
\Bigr)\right) \dd s  \\
- i \eta(t) \int_{0}^t \e^{i (t-s ) \Delta }  \left( \eta\left({s \over \delta}\right) R_K(V^K(s))\right)\dd s.
\end{multline*}
By using again the properties of Bourgain spaces and the estimate
$$
\| f - \Pi_{\tau^{- 1/2}} f \|_{X^{0,b}} \lesssim   \tau^{s_{0}\over 2} \| f \|_{X^{s_{0}, b}},\quad \forall f \in X^{s_{0}, b},
$$
we obtain that
\begin{align*}
\|U^K - V^K \|_{X^{0, b}}
&  \lesssim \delta^{1 \over 4} \| U^K - V^K \|_{X^{0, b}}M_{T}^2 + \| V^K - \Pi_{\tau^{-1/2}} V^K \|_{X^{0, b}}^2 \| V^K\|_{X^{0, b}}\\
& \lesssim  \delta^{1 \over 4} \| U^K - V^K \|_{X^{0, b}}M_{T}^2  + \tau^{s_{0}} M_{T}^3
\end{align*}
and we can conclude as we did before.
\end{proof}

We shall need the following corollary about the propagation of higher regularity with respect to the $b$ parameter.

\begin{cor}\label{corsb}
{
Let  $b \in (5/8, 1]$ and assume that  $s_{0}>0$, then for every    $0 \leq s_{0}' <s_{0},$ we also have that  for every $T > 0$ and uniformly in $\tau$,
$$
\|u^K \|_{X^{s_{0}',b}(T)} \leq M_{T}.
$$}
\end{cor}

\begin{proof}
{
Since $u^K$ solves \eqref{nls}, we have that
$$
u^K(t) = \e^{it \partial_{x}^2} \Pi_{K}u_{0}+ \int_{0}^t \e^{i(t-s) \partial_{x}^2} F(u^K(s))\dd s,
$$
where we set for short
\begin{align*}
i F(u^K) = 2 \Pi_K \left( \Pi_{\taul} u^K \,\Pi_{\taul} \overline u^K\, \Pi_{K^+} u^K \right)
+ \Pi_K \left( \Pi_{\taul} u^K \,\Pi_{\taul} u^K\, \Pi_{K}  \overline  u^K \right).
\end{align*}
Let $u^{K,\eta}$ denote the solution of the truncated Duhamel equation
\begin{equation}\label{duhameltruncated}
u^{K,\eta}(t) = \eta(t)  \e^{it \partial_{x}^2} u_{0}+ \eta(t)\int_{0}^t \e^{i(t-s)\partial_{x}^2}  \eta(s) F(u^{K,\eta}(s))\dd s
\end{equation}
that belongs to the global Bourgain space $X^{s_{0}, b}(\mathbb{R} \times \mathbb{T})$ for $b \in (1/2, 5/8)$ as established in the previous proposition. From the same estimates as before, we obtain that
$$
\|u^{K,\eta}\|_{X^{s_{0}',b}} \lesssim \|u_{0}\|_{H^{s_{0}'}} + \| F(u^{K,\eta}) \|_{X^{s_{0}',b-1}}.
$$
In order to estimate $F(u^{K,\eta})$, we first just use that $b- 1 \leq 0$ so that
$$   \| F(u^{K,\eta}) \|_{X^{s_{0}',b-1}} \lesssim  \| \langle \partial_{x} \rangle^{s_{0}'}F(u^{K,\eta}) \|_{L^2(\mathbb{R} \times \mathbb{T})}$$
where $\langle \partial_{x} \rangle^s$ stands for the Fourier multiplier  $\langle k \rangle^s$.
%
%  we interpolate the estimate of Lemma \ref{lemmabourgaindur} with the trivial estimate
%$$
%\| u \|_{L^2} \lesssim \|u \|_{X^{0,0}}.
%$$
%This yields
%$$
%\| u\|_{L^p} \lesssim \|u \|_{X^{0, 1-b}}, \quad  \tfrac1p= \tfrac{4b-1}6
%$$
%and hence by duality
%$$
%\|u \|_{X^{0, b-1}}   \lesssim  \| u\|_{L^{p'}}.
%$$
Then, by using  the generalized Leibniz rule  (which reads, see for example \cite{Kato-Ponce}, 
\begin{equation}
\label{Leibniz} \|\langle \partial_{x} \rangle^s(fg) \|_{L^p} \lesssim \| \langle \partial_{x} \rangle^s  f \|_{L^{p_{1}}} \|g\|_{L^{q_{1}}} +  \|f\|_{L^{p_{2}}} \|
\langle \partial_{x} \rangle^s g \|_{L^{q_{2}}}
\end{equation}
for every $s>0$ and every $p \in (1, \infty), \, p_{1}, \, p_{2}, \, q_{1}, \, q_{2} \in (1, +\infty]$
such that $p^{-1}= p_{1}^{-1} + q_{1}^{-1}= p_{2}^{-1}+ q_{2}^{-1}$)
we get that
\begin{equation}
\label{revisemieux}
\|F(u^{K,\eta})\|_{X^{s_{0}', b-1}} \lesssim \| \langle \partial_{x} \rangle^{s_{0}'}u^{K,\eta}\|_{L^6(\mathbb{R} \times \mathbb{T} ) }^3.
\end{equation}
To conclude, we  use another $X^{s,b}$ space  estimate due to Bourgain \cite{Bour93a}:  for every $\varepsilon>0$ and $b'>1/2$, 
we have the continuous embedding $X^{\varepsilon, b'} \subset L^6(\mathbb{R} \times \mathbb{T})$, that is to say, 
for every $f \in  X^{\varepsilon, b'}$, we have
$$ \| f \|_{L^6(\mathbb{R} \times \mathbb{T})} \lesssim \|f \|_{X^{\varepsilon, b'}}.$$
By using this last estimate in \eqref{revisemieux}, we thus get that
$$ \|F(u^{K,\eta})\|_{X^{s_{0}', b-1}} \lesssim   \|u^{K,\eta}\|_{X^{s_{0}'+\varepsilon, b'}}^3.$$
Since we can choose  $b'$  less than $5/8$ and  $\varepsilon >0$ such that $s_{0}' + \varepsilon \leq s_{0}$,
%Since $W^{ { 8b - 5 \over 36}, 4} \subset L^{3p'}$ we finally obtain
%$$
%\|F(u^{K,\eta}) \|_{X^{s_{0}', b-1}} \lesssim  \|u^{K,\eta}\|_{W^{s_{0}, 4}}^3
%$$
%and hence, by using again Lemma \ref{lemmabourgaindur} that
%$$
%\|F(u^{K,\eta}) \|_{X^{s_{0}', b-1}} \lesssim  \|u^{K,\eta}\|_{X^{s_{0}, {3\over 8}}}^3.
%$$
%This ends the proof. 
the right-hand side is  already controlled thanks to Proposition \ref{propNLSK2}. This ends the proof.}
\end{proof}

%%%%%%%%%%%%%%%%%%%%%%%%%%%%%%%%%%%%%%%%%%%
\section{Discrete Bourgain spaces}
%%%%%%%%%%%%%%%%%%%%%%%%%%%%%%%%%%%%%%%%%%%
\label{sec:discbourg}

For a sequence $(u_{n})_{n \in \mathbb{Z}}$, we shall define its Fourier transform as
\begin{equation}
\label{discretefourier}
\mathcal{F}_{\tau}(u_{n}) (\sigma)= \tau \sum_{m \in \mathbb{Z}} u_{m} \,\e^{ i m \tau \sigma}.
\end{equation}
This defines a periodic function on $ [-\pi/\tau, \pi/\tau]$ and we have the inverse Fourier transform formula
$$
u_m =  \frac1{2\pi} \int_{-{\pi \over \tau}}^{\pi \over \tau}  \mathcal{F}_{\tau}(u_{n})(\sigma) \,\e^{- i m \tau \sigma}\dd \sigma.
$$
With these definitions the Parseval identity reads
$$
\|u_{n}\|_{l^2_{\tau}}= \| \mathcal{F}_{\tau} (u_n) \|_{L^2(-\pi/\tau, \pi/\tau)},
$$
where the norms are defined by
$$
\|u_{n}\|_{l^2_{\tau}}^2 = \tau \sum_{n \in \mathbb{Z}} |u_{n}|^2, \quad
\| \mathcal{F}_{\tau} (u_n) \|_{L^2(-\pi/\tau, \pi/\tau)}^2 = {1 \over 2 \pi} \int_{-{\pi \over \tau}}^{\pi \over \tau} | \mathcal{F}_{\tau} (u_n) (\sigma) |^2
\dd \sigma.
$$
In this section, we write $L^2$ instead of $L^2 (-\pi/\tau, \pi/\tau)$ for short. We stress the fact that this is not the standard way of normalizing the Fourier series.

We then define in a natural way Sobolev spaces $H^b_{\tau}$ of sequences $(u_{n})_{n\in\mathbb{Z}}$ by \\
$$
\| u_{n}\|_{H^b_{\tau}} =  \left\| \langle d_{\tau}( \sigma)\rangle ^b \mathcal F_\tau(u_n) \right \|_{L^2},
$$
with $d_{\tau}(\sigma)=\frac{\e^{i \tau \sigma} - 1}\tau$ so that we have equivalent norms
$$
\| u_{n}\|_{H^b_{\tau}} =\|\langle D_{\tau}\rangle^b u_{n}\|_{l^2_{\tau}},
$$
where the operator $D_{\tau}$ is defined by $ (D_{\tau}(u_{n}))_{n}=\left( \frac{u_{n-1} - u_{n}}\tau\right)_{n}$  since by definition of the Fourier transform
$$
\mathcal{F}_{\tau}( D_{\tau} u_{n}) (\sigma) =d_{\tau}(\sigma) \mathcal{F}_{\tau}(u_{n}) (\sigma).
$$
Note that $d_{\tau}$ is $2\pi/\tau$ periodic and that uniformly in $\tau$, we have $|d_{\tau}(\sigma)| \sim | \sigma |$ for $|\tau \sigma | \leq \pi$.

For sequences of functions $(u_{n}(x))_{n \in \mathbb{Z}},$ we define the Fourier transform $\widetilde{u_{n}}(\sigma, k)$ by
$$
\mathcal F_{\tau,x}(u_n)(\sigma,k) =\widetilde{u_{n}} (\sigma, k)= \tau \sum_{m \in \mathbb{Z}} \widehat{u_{m}}(k) \,\e^{i m \tau \sigma}, \quad \widehat{u_{m}}(k)= {1 \over 2\pi} \int_{-\pi}^\pi u_{m}(x) \,\e^{-i k x}\dd x.
$$
Parseval's identity then reads
\begin{equation}\label{parseval}
\| \widetilde{u_{n}}\|_{L^2l^2}= \|u_{n}\|_{l^2_{\tau}L^2},
\end{equation}
where
$$
\| \widetilde{u_{n}}\|_{L^2l^2}^2 = \int_{-{\pi \over \tau}}^{\pi\over \tau} \sum_{k \in \mathbb{Z}}
|\widetilde{u_{n}}(\sigma, k)|^2 \dd \sigma, \quad
\|u_{n}\|_{l^2_{\tau}L^2}^2 = \tau \sum_{m \in \mathbb{Z}} \int_{-\pi}^\pi  |u_{m}(x)|^2 \dd x.
$$
We then define the discrete Bourgain spaces $X^{s,b}_\tau$ for $s\ge 0$, $b\in\mathbb R$, $\tau>0$ by
\begin{equation}\label{bourgdef1}
\| u_n \|_{X^{s,b}_{\tau}}= \| \e^{-i n \tau \partial_{x}^2 } u_{n}\|_{H^b_{\tau} H^s }=  \|
\langle D_{\tau}\rangle^b \langle \partial_{x} \rangle^s (\e^{-i n \tau \partial_{x}^2 } u_{n})\|_{l^2_{\tau} L^2 }.
\end{equation}
As in the continuous case, we obtain the following properties.

\begin{lemma}\label{lemequiv}
With the above definition, we have that
\begin{equation}\label{norm2}
\| u_n \|_{X^{s,b}_{\tau}} \sim  \left\| \langle k \rangle^s \langle  d_{\tau}(\sigma -k^2)  \rangle^b \widetilde{u_n}(\sigma, k)  \right\|_{L^2l^2}.
\end{equation}
Moreover, for $s \in \mathbb{R}$ and $b> 1/2$, we have that $X^{s,b}_{\tau} \subset l^\infty_{\tau}H^s$:
\begin{equation}\label{sobbourg}
\|u_{n}\|_{l^\infty_{\tau}H^s } \lesssim_{b} \| u_{n}\|_{X^{s,b}_{\tau}}.
\end{equation}
\end{lemma}
The weight $d_{\tau}(\sigma -k^2)$ obviously vanishes if $\tau(\sigma - k^2)= 2 m \pi$ for $m \in \mathbb{Z}$. For a localized function such that $k$ is constrained to $|k| \lesssim  \tau^{-{1 \over 2}}$ this will behave like in the continuous case with only a cancellation when $\sigma = k^2$. 
For larger frequencies, however, there are additional cancellations that will create some loss in the product estimates.

Note that the seemingly different behavior that we have here in the discrete case compared with
the definition \eqref{defnormeBourgain} in the continuous case comes from our definition \eqref{discretefourier} of the discrete Fourier transform.
Let us recall that in the continuous case 
 $$\|u \|_{X^{s,b}}=  \| \overline{u} \|_{X^{s,b}_{\sigma = k^2}},$$
 where $ \| \overline{u} \|_{X^{s,b}_{\sigma = k^2}}= \| \langle k \rangle^s \langle \sigma - k^2 \rangle^b \tilde u \|_{L^2(\mathbb{R}\times
 \mathbb{Z})}$
 so that we can easily deduce the properties of $X^{s,b}_{\sigma = k^2}$  from the properties of $X^{s,b}$.

\begin{proof}
Let us set $ f_{n}(x)= \e^{-i n \tau \partial_{x}^2} u_{n}(x)$. From the definition of $\mathcal F_\tau$, we get that
$$
\widetilde{f_n}(\sigma, k)= \tau \sum_{m \in \mathbb{Z}} \widehat{u_{m}}(k) \,\e^{ i m \tau ( \sigma  + k^2)}
$$
so that
\begin{equation}\label{transl}
\widetilde{f_n}(\sigma, k) = \widetilde{u_n}( \sigma + k^2, k).
\end{equation}
Therefore,
$$
\| u_n \|_{X^{s,b}_{\tau}}^2 =  \sum_{k \in \mathbb{Z}} \int_{-{\pi \over \tau}}^{\pi \over \tau}
\langle d_{\tau}(\sigma) \rangle^{2b} \langle k \rangle^{2s} |\widetilde{u_n}( \sigma + k^2, k )|^2 \dd \sigma
$$
and the result follows by a change of variables.

To prove the embedding \eqref{sobbourg}, it suffices to prove that
$$
\| f_{n} \|_{l^\infty_{\tau}H^s} \lesssim \|f_{n}\|_{H^b_{\tau}H^s}.
$$
Since
$$
\widehat{f_n}(k)= \int_{-{\pi \over \tau}}^{\pi \over \tau} \widetilde{f_m}(\sigma, k) \e^{-in \tau \sigma}\dd \sigma,
$$
we get from Cauchy--Schwarz that
$$
| \widehat{f_n}(k)| \lesssim  \left(  \int_{-{\pi \over \tau}}^{\pi \over \tau}  { 1 \over \langle d_{\tau} (\sigma ) \rangle^{2b}}
\dd \sigma \right)^{1 \over 2} \| \langle d_{\tau}(\sigma)  \rangle^{b}  \widetilde{f_m}( \sigma, k)\|_{L^2}.
$$
The result then follows by multiplying the above inequality by $\langle k \rangle^s$ and taking the $L^2$ norm with respect to $k$.
\end{proof}

\begin{rem}\label{remshift}
From Lemma \ref{lemequiv}, we can make the following useful observation
\begin{equation}\label{bshift}
\sup_{\delta \in [-4, 4]} \| \e^{i n\tau \delta \partial_{x}^2} u_{n}\|_{X^{s, b}_{\tau}} \lesssim_{b} \|u_{n}\|_{X^{s,b}_{\tau}}.
\end{equation}
Note that this follows at once from $|d_{\tau}(\sigma - k^2 + \delta)| \lesssim  \langle d_{\tau}(\sigma - k^2) \rangle$.
\end{rem}

\begin{rem}\label{remembedd}
 Since $\langle d_{\tau}(\sigma) \rangle \lesssim { 1 \over  \tau}$, the discrete spaces satisfy the embedding
\begin{equation}\label{embdisc1}
\|u_{n}\|_{X^{0, b}_{\tau}} \lesssim { 1 \over  \tau^{b-b'}} \|u_{n}\|_{X^{0, b'}_\tau}, \quad b \geq b'.
\end{equation}
Indeed, from the above observation we get that the inequality is true for $b \geq 0$, $b'=0$. Next, by interpolation we obtain the case $b \geq b' \geq 0$. The case $0 \geq b \geq b'$ then follows by duality and the general case by composition.
\end{rem}

We shall now establish the counterpart of Lemma \ref{lemmabourgainfacile} at the discrete level.

\begin{lemma}\label{bourgainfaciled}
For $\eta \in \mathcal{C}^\infty_{c}(\mathbb{R})$ and $\tau\in(0,1]$, we have that
\begin{align}
\label{bourg1} &\| \eta(n \tau)  \e^{in \tau \partial_{x}^2} f\|_{X^{s,b}_{\tau}} \lesssim_{\eta, b} \|f\|_{H^s}, \quad s \in \mathbb{R}, \, b \in \mathbb{R}, \, f \in H^s, \\
\label{bourg2} &\| \eta(n \tau)  u_{n}\|_{X^{s,b}_{\tau}} \lesssim_{\eta, b} \|u_{n}\|_{X^{s,b}_{\tau}}, \quad s \in \mathbb{R}, \, b \in \mathbb{R} , \, u_{n} \in X^{s,b}_{\tau},\\
\label{bourg3} &\left\| \eta\left(\frac{n\tau}T \right) u_{n} \right\|_{X^{s,b'}_{\tau}} \lesssim_{\eta, b, b'} T^{b-b'} \|u_{n}\|_{X^{s,b}_{\tau}}, \quad s \in \mathbb{R},  -{1 \over 2} <b' \leq b <{ 1 \over 2},\, 0< T = N \tau  \leq 1, \, N \geq 1.
\end{align}
In addition, for
$$
U_{n}(x)= \eta(n \tau) \tau \sum_{m=0}^n  \e^{i ( n-m ) \tau \partial_{x}^2}  u_{m}(x),
$$
we have
\begin{equation}
\label{bourg4}\|U_{n}\|_{X^{s,b}_{\tau}} \lesssim_{\eta, b} \|u_{n}\|_{X^{s, b-1}_{\tau}}, \quad s \in \mathbb{R}, \, b>1/2.
\end{equation}
We stress that all given estimates are uniform in $\tau$.
\end{lemma}

\begin{proof}
We begin with \eqref{bourg1}.  Let us set $u_{n}(x)= \eta(n \tau) e^{in \tau \partial_{x}^2} f(x)$. We first observe that
$$
\widetilde{u_n}(\sigma, k) = \mathcal{F}_{\tau}( \eta(n \tau))( \sigma -k^2) \widehat f(k).
$$
The function $g(\sigma)=  \mathcal{F}_{\tau}( \eta(n \tau))(\sigma)$  is fastly decreasing in the sense that
\begin{equation}\label{gfast}
\left| d_{\tau}(\sigma)^L g(\sigma) \right| \lesssim 1,
\end{equation}
where the estimate is uniform in $\tau$ and $\sigma $ for every {integer $L\ge 1$}. Indeed we have that
$$
d_{\tau}(\sigma) g(\sigma)=\tau \sum_{n \in \mathbb{Z}} { \eta((n-1) \tau)- \eta(n \tau) \over \tau} \,\e^{in \tau \sigma}
$$
and therefore
\begin{equation}\label{unif}
\left| d_{\tau}(\sigma) g(\sigma)  \right| \lesssim 1
\end{equation}
follows from the smoothness of $\eta$. We easily get the boundedness of higher powers by induction. The estimate then follows easily from Lemma \ref{lemequiv}.

Let us prove \eqref{bourg2}. We recall that
$$
\eta(n \tau)= \frac1{2\pi}\int_{-{\pi \over \tau}}^{\pi \over \tau}  g(\sigma) \,\e^{-in \tau \sigma} \dd \sigma.
$$
We deduce from \eqref{gfast} that for every $L \geq 1$, there exists $C>0$ such that for every $\tau \in (0, 1]$ and $\sigma$ with $\tau \sigma \in [-\pi, \pi]$,
\begin{equation}\label{gfast2}
|g(\sigma)| \leq {C \over  \langle \sigma \rangle^L}.
\end{equation}
This yields, by using the fast decay of $g(\sigma)$, that
$$
\|  \eta(n \tau) u_{n} \|_{X^{s, b}_{\tau}} \lesssim\int_{-{\pi \over \tau}}^{\pi \over \tau} { 1 \over \langle \sigma_{0} \rangle^L}  \|u_{n} \,\e^{-in \tau \sigma_{0}} \|_{X^{s, b}_{\tau}} \dd \sigma_{0}.
$$
Next, since $\mathcal{F}_{\tau,x} (u_{n} \e^{-in \tau \sigma_{0}})(\sigma, k)= \widetilde{u_n}(\sigma -\sigma_{0}, k)$ and
$\langle  d_{\tau}(\sigma-\sigma_{0} - k^2)  \rangle^b \lesssim  \langle \sigma_{0} \rangle^{|b|}   \langle  d_{\tau}(\sigma- k^2)  \rangle^b$,
we get that
$$
\|  \eta(n \tau) u_{n} \|_{X^{s, b}_{\tau}} \leq  \int_{-{\pi \over \tau}}^{\pi \over \tau} { 1 \over \langle \sigma_{0} \rangle^{L -|b|}}\dd \sigma_{0} \,  \|u_{n}  \|_{X^{s, b}_{\tau}} \lesssim  \|u_{n}  \|_{X^{s, b}_{\tau}}
$$
by choosing $L$ sufficiently large.

We turn to the proof of \eqref{bourg3}. We follow the steps of the  proof of the  continuous case in \cite{Tao06}. We observe that by composition it suffices to handle the cases $0 \leq b' \leq b$ or $b' \leq b \leq 0$. By duality, it suffices then to establish the inequality in the case $0 \leq b' \leq b$. By standard interpolation, we have that
$$
\left\| \eta\Bigl({ n \tau \over T}\Bigr) u_{n} \right\|_{X^{s,b'}_{\tau}} \leq  \left\| \eta\Bigl({ n \tau \over T}\Bigr) u_{n} \right\|_{X^{s,0}_{\tau}}^{ 1- {b'\over b}}  \left\| \eta\Bigl({ n \tau \over T}\Bigr) u_{n} \right\|_{X^{s,b}_{\tau}}^{b' \over b}.
$$
It thus suffices to prove that
\begin{equation}\label{trunc1}
\left\| \eta\Bigl({ n \tau \over T}\Bigr) u_{n} \right\|_{X^{s,b}_{\tau}} \lesssim  \|  u_{n} \|_{X^{s,b}_{\tau}}
\end{equation}
and  that
\begin{equation}\label{trunc2}
\left\| \eta\Bigl({ n \tau \over T} \Bigr) u_{n} \right\|_{X^{s,0}_{\tau}} \leq T^b   \left\| u_{n} \right\|_{X^{s,b}_{\tau}}
\end{equation}
 for $b<1/2$, where the estimates are uniform for $T \in (0, 1]$. We start with the first estimate. Note that we cannot use directly \eqref{bourg2} to get an estimate uniform in $T$. Let us set  $f_{n}= \e^{-in\tau \partial_{x}^2} u_{n}$ and $U_{n}= \eta(n \tau /T) f_{n}$. We want to estimate
$$
\left\| \eta\Bigl({ n \tau \over T}\Bigr) f_{n} \right\|_{H^b_{\tau}H^s}= \|U_{n}\|_{H^b_{\tau}H^s}.
$$
We have that
$$
\widetilde{U_m}(\sigma, k)=  \frac1{2\pi}\int_{-{\pi \over \tau}}^{\pi \over \tau} g_{T}(\sigma - \sigma') \widetilde{f_m}(\sigma', k)\dd \sigma',
$$
where we have set
$$
g_{T}(\sigma)= \tau \sum_{n} \eta\Bigl({n\tau \over T}\Bigr) \,\e^{in \tau \sigma}.
$$
Using the same argument as above, we observe that for every $L \geq 0$
\begin{equation}\label{gdecay}
|g_{T}(\sigma)| \lesssim_{L} { T \over \langle T \sigma \rangle^L}.
\end{equation}
In particular, this yields
\begin{equation}\label{gT}
\|g_{T}\|_{L^1(-\pi/\tau, \pi/\tau)} \lesssim 1, \quad  \| \langle \sigma \rangle^b g_{T}\|_{L^2(-\pi/\tau, \pi/\tau)} \lesssim  T^{{1 \over 2} - b}, \quad  \| \langle \sigma \rangle^b g_{T}\|_{L^1(-\pi/\tau, \pi/\tau)} \lesssim T^{-b}.
\end{equation}
We can first write by using Young's inequality for convolutions 
$$
\left\| \langle d_{\tau}(\sigma)\rangle^b \widetilde{U_m}(\sigma, k)  \right\|_{L^2} \lesssim \| g_{T}\|_{L^1} \| \langle d_{\tau} \rangle^b \widetilde{f_m}(\cdot, k)\|_{L^2}
+ \left \| \int_{-{\pi \over \tau}}^{\pi \over \tau}| \langle d_{\tau}(\sigma - \sigma') \rangle^b g_{T}(\sigma - \sigma') |\widetilde{f_m}(\sigma', k)|\dd \sigma' \right\|_{L^2}.
$$
To estimate the last integral, we split $\widetilde{f_m}(\sigma',  k) =  \widetilde{f_m}(\sigma', k) \mathrm{1}_{|\sigma' T | \leq 1} + \widetilde{f_m}(\sigma', k) \mathrm{1}_{|\sigma' T | \geq 1}$.
For the first contribution, we write
\begin{multline*}
\left \| \int_{-{\pi \over \tau}}^{\pi \over \tau}| \langle d_{\tau}(\sigma - \sigma') \rangle^b g_{T}(\sigma - \sigma') |\widetilde{f_m}(\sigma', k)| \mathrm{1}_{|\sigma' T | \leq 1}\dd \sigma' \right \|_{L^2}\\
\lesssim  \left\| \langle \sigma \rangle^b g_{T}\right\|_{L^2(-\pi/\tau, \pi/\tau)} \left\| \mathrm{1}_{|\sigma T | \leq 1} \widetilde{f_m}(\sigma,k) \right\|_{L^1} \lesssim T^{{1 \over 2} - b} T^{ b- { 1 \over 2}} \left\| \widetilde{f_m}(\cdot, k) \right\|_{L^2},
\end{multline*}
where we have used Cauchy--Schwarz to get the last estimate. For the second contribution, we use
\begin{multline*}
\left \| \int_{-{\pi \over \tau}}^{\pi \over \tau}| \langle d_{\tau}(\sigma - \sigma') \rangle^b g_{T}(\sigma - \sigma') |\widetilde{f_m}(\sigma', k)| \mathrm{1}_{|\sigma' T | \geq 1}\dd \sigma' \right \|_{L^2}\\
\lesssim  \left\| \langle \sigma \rangle^b g_{T}\right\|_{L^1(-\pi/\tau, \pi/\tau)} \left\| \mathrm{1}_{|\sigma T | \geq  1} \widetilde{f_m}(\sigma,k) \right\|_{L^2} \lesssim T^{ - b} \left\| \mathrm{1}_{ |\sigma T| \geq 1} \widetilde{f_m}(\cdot, k) \right\|_{L^2}
\lesssim   \left\| \langle \sigma \rangle^b \widetilde{f_m}(\cdot, k) \right\|_{L^2},
\end{multline*}
where we have now used that for $|\sigma T  |\geq 1$, $ T^{-b} \lesssim \langle \sigma \rangle^b$. We have thus obtained that
$$
\| \langle d_{\tau} \rangle^b \widetilde{U_m}(\cdot, k) \|_{L^2} \lesssim \| \langle d_{\tau} \rangle^b \widetilde{f_m}(\cdot, k) \|_{L^2}.
$$
It suffices to multiply by $\langle k \rangle^s$ and to take the $L^2$ norm in $k$ to get \eqref{trunc1}.

We next prove \eqref{trunc2}. Again, it suffices to prove that
$$
\| U_{n}\|_{l^2_{\tau}H^s} \lesssim T^{b} \|f_{n}\|_{H^b_{\tau}H^s}.
$$
To establish this estimate, we split $f_{n}= f_{n,1}+ f_{n,2}$ with
$$
\widetilde{f_{m,1}}(\sigma, k)= \widetilde{f_m}(\sigma,k) \mathrm{1}_{ |T \sigma| \geq 1}, \quad \widetilde{f_{m,2}}(\sigma, k)= \widetilde{f_m}(\sigma,k) \mathrm{1}_{ |T \sigma| \leq 1}.
$$
For the first part, we readily obtain from the definition of the norm that
$$
\| f_{n,1}\|_{l^2_{\tau} H^s} \lesssim T^b \|f_n\|_{H^b_{\tau}, H^s}
$$
since $1 \lesssim T^b | \sigma |^b$ on the support of integration. Since $\eta$ is bounded
$$
\left\| \eta\Bigl({n \tau \over T}\Bigr)f_{n,1}\right\|_{l^2_{\tau}H^s} \lesssim T^b \|f_{n}\|_{H^b_{\tau}H^s}.
$$
For the other part, we use that
$$
\widehat{f_{n,2}}(k)=  \int_{-{\pi \over \tau}}^{\pi \over \tau}\,\e^{-in \tau \sigma} |\sigma|^{-b}\mathrm{1}_{T| \sigma | \leq 1}|\sigma|^b \widetilde{f_m}(\sigma, k)\dd \sigma.
$$
This yields from Cauchy--Schwarz that
$$
|\widehat{f_{n,2}}(k) |^2 \lesssim T^{2b- { 1}} \left(\int_{-{\pi \over \tau}}^{\pi \over \tau} \langle d_{\tau}(\sigma)\rangle^{2b} | \widetilde{f_m}(\sigma, k)|^2 \dd \sigma\right)
$$
and therefore, for every $n$, we have
$$
\| f_{n,2}\|_{H^s}^2 =\| \langle k \rangle^s\widehat{f_{n,2}} \|_{l^2}^2 \lesssim  T^{2b- { 1}}  \|f_m\|_{l^2_{\tau}H^s}^2.
$$
This yields
$$
\left\| \eta\left({n \tau \over T}\right) f_{n, 2}\right\|_{l^2_{\tau}H^s}^2
= \tau \sum_{n}  \eta\Bigl({n \tau \over T}\Bigr)^2  \|f_{n,2}\|_{H^s}^2 \lesssim
\tau \sum_{n} \eta\Bigl({n \tau \over T}\Bigr)^2  T^{2b- { 1}}  \|f_m\|_{l^2_{\tau}H^s}^2
\lesssim  T    T^{2b- { 1}}  \|f_m\|_{l^2_{\tau}H^s}^2
$$
and we get \eqref{trunc2}, which concludes the proof of \eqref{bourg3}.

We finally prove \eqref{bourg4}. Let us set
$$
F_{n}(x)= \e^{-in \tau \partial_{x}^2} U_{n}(x), \quad f_{n}(x) = \e^{-in \tau \partial_{x}^2} u_{n}(x)
$$
so that
$$
F_{n}(x)= \eta(n\tau)  \tau \sum_{m=0}^n f_{m}.
$$
It suffices to prove that
$$
\|F_{n}\|_{H^b_{\tau}H^s} \lesssim \| f_{n}\|_{H^{b-1}_{\tau}H^s}.
$$
We shall only prove the estimate for $s=0$. The general case just follows by applying $\langle\partial_{x}\rangle^s$. Let us use again the function $g(\sigma)= \mathcal{F}_{\tau}(\eta (n \tau))(\sigma)$ as above. By direct computation, we find that
$$
\widehat{F_n}(k)= \eta (n\tau)  \tau \int_{-{\pi \over \tau}}^{\pi \over \tau} \widetilde{f_m}(\sigma_{0}, k) {1- \e^{ -i(n+1) \tau \sigma_{0}} \over 1 - \e^{-i \tau \sigma_{0}} }  \dd\sigma_{0}
$$
and therefore
$$
\widetilde{F_m}(\sigma, k)= \int_{-{\pi \over \tau}}^{\pi \over \tau} {\e^{i \tau \sigma_{0}} \over d_{\tau}(\sigma_{0})} \widetilde{f_m}(\sigma_{0}, k)
    \left(g(\sigma)- \e^{-i \tau \sigma_{0}} g(\sigma - \sigma_{0})\right)\dd \sigma_{0}.
$$
We then split
$$
\widetilde{F_m}(\sigma, k)= \widetilde{F_{m,1}}(\sigma, k)+ \widetilde{F_{m, 2}}(\sigma, k),
$$
where we replace $\widetilde{f_m}$ by $\widetilde{f_m}(\sigma_{0}, k) \mathrm{1}_{| \sigma _{0}| \geq 1}$ in $\widetilde{F_{m,1}}(\sigma, k)$, and $\widetilde{f_m}$ by $\widetilde{f_m}(\sigma_{0}, k) \mathrm{1}_{| \sigma _{0}| \leq 1}$ in $\widetilde{F_{m,2}}(\sigma, k)$. By using again that $g$ has a fast decay \eqref{gdecay}, this yields
\begin{multline*}
\langle d_{\tau}\rangle^b  | \widetilde{F_{m, 1}}(\sigma, k) |    \lesssim \langle \sigma \rangle^{b-L}   \left( \int_{-{\pi \over \tau}}^{\pi \over \tau}
   {  1 \over   \langle d_{\tau}(\sigma_{0})\rangle^{  2 b}} \dd \sigma_{0} \right)^{ 1 \over 2}
    \| \langle d_{\tau}\rangle^{b-1} \widetilde{f_m}(\cdot, k) \|_{L^2}\\
+  \int_{-{\pi \over \tau}}^{\pi \over \tau}  \langle d_{\tau}(\sigma_{0}) \rangle^{b-1}  |\widetilde{f_m}(\sigma_{0}, k)|
      \langle \sigma - \sigma_{0} \rangle^{b-L}\dd \sigma_{0}.
 \end{multline*}
Therefore, by taking the $L^2$ norm in $\sigma$, and by using Young's inequality for convolutions for the second term, we obtain since $b>1/2$,
$$
\| \langle d_{\sigma} \rangle^b \widetilde{F_{m,1}}(\cdot, k)\|_{L^2} \lesssim \| \widehat{f_m}(k)\|_{H^{b-1}_{\tau}}.
$$
To estimate $\widetilde{F_{m,2}}$, we observe that for $|\sigma_{0}| \leq 1$, we can use Taylor's formula to get that
$$
\left|  { \langle d_{\tau}(\sigma_{0}) \rangle^b \over d_{\tau}(\sigma_{0})} \left(g(\sigma)- \e^{-i \tau \sigma_{0}} g(\sigma - \sigma_{0})\right)  \right|
\lesssim { 1 \over \langle \sigma - \sigma_{0} \rangle^L}.
$$
The estimate then follows from the same arguments.

We have thus proven that
$$
\| \langle d_{\sigma} \rangle^b \widetilde{F_m}(\cdot,k)\|_{L^2} \lesssim \| \widehat{f_m}(k)\|_{H^{b-1}_{\tau}}.
$$
To conclude, it suffices to take the $L^2$ norm with respect to $k$.
\end{proof}

\begin{rem}
Note that in the proof of \eqref{bourg2},  we have also  established a useful time translation invariance property of the discrete Bourgain spaces, that is to say
\begin{equation}\label{shiftt}
\sup_{\delta \in [-4, 4]} \|\e^{i n \tau \delta } u_{n}\|_{X^{s, b}_{\tau}} \lesssim_{b} \|u_{n}\|_{X^{s, b}_{\tau}}.
\end{equation}
\end{rem}

We shall finally study in this section the discrete counterpart of Lemma \ref{lemmabourgaindur} which is crucial for the analysis of nonlinear problems.

In the discrete setting, for a sequence $(u_{n}) \in l^p(\mathbb{Z}, X)$, with $X$ normed space  we use the norm
\beq
\label{normelptau} \|u_{n}\|_{l^p_{\tau} (X)}=\left( \tau \sum_{n \in \mathbb{Z}} \|u_{n}\|_{X}^p \right)^{1 \over p}.
\eeq

\begin{lemma}\label{prodd}
For $K \geq \tau^{-{ 1 \over 2}}$,  we have
\begin{equation}\label{prodd1}
\|\Pi_{K}u_{n} \|_{l^4_{\tau}L^4} \lesssim  (K \tau^{ 1 \over 2})^{1 \over 2} \|u_{n}\|_{X^{0, {3\over 8}}_{\tau}}.
\end{equation}
%    Moreover, we also have  if $K= \tau^{ - { \gamma \over 2} }$, $\gamma \geq 1$,
%     \begin{equation}
%   \label{prodd2} \|\Pi_{K}u_{n} \|_{l^4_{\tau}L^4} \lesssim  \|u_{n}\|_{X^{{ 1 \over 2} (1- { 1 \over \gamma}), {3\over 8}}_{\tau}}.
%    \end{equation}
\end{lemma}
%  Note that the estimate \eqref{prodd2} is always better for $\gamma \leq 2$ than the Sobolev embedding $ H^{1 \over 4} \subset L^4$ in
%  dimension $1$.

The above inequality is  an  important result  of the paper, but the understanding of its  proof that requires tools not yet introduced is not necessary to
continue reading the paper. For the convenience of the reader, we thus postpone it to section \ref{sectiontechnic}.

By duality, we also get  from \eqref{prodd1} that
\begin{equation}\label{duald}
\| \Pi_{K} u_{n}\|_{X^{0, -{3 \over 8}}_{\tau}} \lesssim (K \tau^{1 \over 2})^{1 \over 2} \|u_{n}\|_{l^{4\over 3}_{\tau}L^{4\over 3}}.
\end{equation}

As a consequence, we obtain the following crucial product estimates for sequences $u_n$, $v_n$, and $w_n$.

\begin{cor}\label{corprod}
We have the following product estimate:
\begin{equation}\label{prodd3}
\| \Pi_{K}( \Pi_{ \tau^{- {1 \over 2 } }} u_{n} \Pi_{ \tau^{- {1 \over 2 } }} v_{n} \Pi_{K} w_{n} ) \|_{X^{0, - { 3 \over 8}}_{\tau}}
\lesssim   (K \tau^{1 \over 2}) \|u_{n}\|_{X^{0, {3 \over 8}}_{\tau}}  \|v_{n}\|_{X^{0, {3 \over 8}}_{\tau}} \|w_{n}\|_{X^{0, {3 \over 8}}_{\tau}}.
\end{equation}
Moreover, for any $s_{1}>1/4$,  we have the estimates
\begin{align}\label{prodd3ter}
\| \Pi_{K}( \Pi_{ \tau^{- {1 \over 2 } }} u_{n} \Pi_{ \tau^{- {1 \over 2 } }} v_{n} \Pi_{K} w_{n} ) \|_{X^{0, - { 3 \over 8}}_{\tau}} &\lesssim (K \tau^{1 \over 2})^{1 \over 2}
\|u_{n}\|_{X^{0, {3 \over 8}}_{\tau}}  \|v_{n}\|_{X^{0, {3 \over 8}}_{\tau}} \|w_{n}\|_{l^4_{\tau}H^{s_{1}}},\\
\label{prodd3bis}
\| \Pi_{K}( \Pi_{ \tau^{- {1 \over 2 } }} u_{n} \Pi_{ \tau^{- {1 \over 2 } }} v_{n} \Pi_{K} w_{n} ) \|_{X^{0, - { 3 \over 8}}_{\tau}}&\lesssim
\|u_{n}\|_{X^{s_{1}, {3 \over 8}}_{\tau}}  \|v_{n}\|_{X^{s_{1}, {3 \over 8}}_{\tau}} \|w_{n}\|_{l^\infty_{\tau}L^2},
\end{align}
and for $s_{2}>1/2$, we have
\begin{align}\label{prodd3bisbis}
\| \Pi_{K}( \Pi_{ \tau^{- {1 \over 2 } }} u_{n} \Pi_{ \tau^{- {1 \over 2 } }} v_{n} \Pi_{K} w_{n} ) \|_{X^{0, - { 3 \over 8}}_{\tau}} &  \lesssim \|u_{n}\|_{X^{0, {3 \over 8}}_{\tau}}  \|\Pi_{ \tau^{- {1 \over 2 } }} v_{n}\|_{l^4_{\tau}H^{s_{1}}} \|w_{n}\|_{l^\infty_{\tau}H^{s_{2}}}, \\
\label{prodd3bisbis2}
\| \Pi_{K}( \Pi_{ \tau^{- {1 \over 2 } }} u_{n} \Pi_{ \tau^{- {1 \over 2 } }} v_{n} \Pi_{K} w_{n} ) \|_{X^{0, - { 3 \over 8}}_{\tau}} &
\lesssim  \|u_{n}\|_{X^{0, {3 \over 8}}_{\tau}}  \| v_{n}\|_{X^{0, {3 \over 8}}_{\tau}} \|w_{n}\|_{l^\infty_{\tau}H^{s_{2}}}.
\end{align}

% for  $ K= \tau^{ - { \alpha  \over 2} }$, $\alpha \geq 1$,
%  let us set $s_{1}= s_{0}''- ({  1 - { 1 \over \alpha}})$, and assume that $s_{1} \geq 0$. Then
%  \begin{equation}
%  \label{prodd4} \| \Pi_{K}( \Pi_{ \tau^{- {1 \over 2 } }} u_{n} \Pi_{ \tau^{- {1 \over 2 } }} v_{n} \Pi_{K} w_{n} ) \|_{X^{s_{1}, - { 3 \over 8}}_{\tau}}
%   \lesssim  \|u_{n}\|_{X^{s_{1}, {3 \over 8}}_{\tau}}   \|v_{n}\|_{X^{s_{1}, {3 \over 8}}_{\tau}}
%    \|w_{n}\|_{X^{s_{0}'', {3 \over 8}}_{\tau}}.
%    \end{equation}
%    Moreover, we also have that
% \begin{equation}
%  \label{prodd6} \| \Pi_{K}( \Pi_{ \tau^{- {1 \over 2 } }} u_{n} \Pi_{ \tau^{- {1 \over 2 } }} v_{n} \Pi_{K} w_{n} ) \|_{X^{0, - { 3 \over 8}}_{\tau}}
%   \lesssim   (K \tau^{1 \over 2})^{1 \over 2} \|u_{n}\|_{X^{{1 \over 8}, {3 \over 8}}_{\tau}}  \|v_{n}\|_{X^{{1 \over 8}, {3 \over 8}}_{\tau}}
%    \|w_{n}\|_{X^{0, {3 \over 8}}_{\tau}}
%    \end{equation}
% \begin{equation}
%  \label{prodd5} \| \Pi_{K}( \Pi_{ \tau^{- {1 \over 2 } }} u_{n} \Pi_{ \tau^{- {1 \over 2 } }} v_{n} \Pi_{K} w_{n} ) \|_{X^{s_{1}, - { 3 \over 8}}_{\tau}}
%   \lesssim  \|u_{n}\|_{X^{s_{1}+ {1 \over 8}, {3 \over 8}}_{\tau}}   \|v_{n}\|_{X^{s_{1}+ {1 \over 8}, {3 \over 8}}_{\tau}}
%    \|w_{n}\|_{X^{s_{1}+ {1 \over 2}({ 1- {1 \over \alpha} )}, {3 \over 8}}_{\tau}}.
%    \end{equation}
\end{cor}
% Note that the difference between \eqref{prodd4} and \eqref{prodd5}, is that in  \eqref{prodd4}, we loose
%  derivatives only on the third function but we loose $ 1 - {1 \over \alpha}$ derivatives as in   \eqref{prodd5},
%  we loose only  ${1 \over 2}  (1 - {1 \over \alpha})$ derivative on the last function but we also loose $1/8$ on the
%  two first ones. As far as the total loss of derivatives is concerned, the second estimate is better than the first one
%  if $\alpha \geq 8/7.$ The difference between the  estimates \eqref{prodd6}  and \eqref{prodd3}
%   is that  \eqref{prodd3} will allow to handle rougher data but will give a worse error estimate in case of smoother ones.

 Note that \eqref{prodd3bis} is of particular interest if the two lower frequency factors have at least $1/4$ regularity. Then we do not need the factor $K \tau^{1/2}$ which is large if $\alpha >1$ (recall that $K=\tau^{-\alpha/2}$). This will be useful to prove the stability of the scheme for $\alpha >1$. The estimate \eqref{prodd3ter} will turn out to be useful to optimize the convergence rate of the scheme when $s_{0}$ is large enough.

\begin{proof}
We start with proving \eqref{prodd3}. We first obtain from the estimate \eqref{duald} that
$$
\| \Pi_{K}( \Pi_{ \tau^{- {1 \over 2 } }} u_{n} \Pi_{ \tau^{- {1 \over 2 } }} v_{n} \Pi_{K} w_{n} ) \|_{X_\tau^{0, - { 3 \over 8}}} \lesssim  (K \tau^{1 \over 2})^{1 \over 2}  \| \Pi_{K}( \Pi_{ \tau^{- {1 \over 2 } }} u_{n} \Pi_{ \tau^{- {1 \over 2 } }} v_{n} \Pi_{K} w_{n} ) \|_{l^{4\over 3}_{\tau}L^{4\over 3}}.
$$
From the continuity of $\Pi_{K}$ on $L^p$ and the H\"{o}lder inequality, we next get that
\begin{equation}\label{eq:prod-inter}
\| \Pi_{K}( \Pi_{ \tau^{- {1 \over 2 } }} u_{n} \Pi_{ \tau^{- {1 \over 2 } }} v_{n} \Pi_{K} w_{n} ) \|_{X_\tau^{0, - { 3 \over 8}}} \lesssim (K \tau^{1 \over 2})^{1 \over 2}   \| \Pi_{ \tau^{- {1 \over 2 } } }u_{n} \|_{l^4_{\tau}L^4}  \| \Pi_{ \tau^{- {1 \over 2 } } }v_{n} \|_{l^4_{\tau}L^4} \| \Pi_{K} w_{n} \|_{l^4_{\tau}L^4}.
\end{equation}
By using  again \eqref{prodd1}, we thus find
$$
\| \Pi_{K}( \Pi_{ \tau^{- {1 \over 2 } }} u_{n} \Pi_{ \tau^{- {1 \over 2 } }} v_{n} \Pi_{K} w_{n} ) \|_{X_\tau^{0, - { 3 \over 8}}} \lesssim   (K \tau^{1 \over 2}) \|u_{n}\|_{X^{0, {3 \over 8}}_{\tau}}  \|v_{n}\|_{X^{0, {3 \over 8}}_{\tau}}\|w_{n}\|_{X^{0, {3 \over 8}}_{\tau}}.
$$
This proves \eqref{prodd3}.

For the proof of \eqref{prodd3ter}, we use again \eqref{eq:prod-inter}. However, we only estimate $\| \Pi_{ \tau^{- {1 \over 2 } } }u_{n} \|_{l^4_{\tau}L^4}$ and $\| \Pi_{ \tau^{- {1 \over 2 } } }v_{n} \|_{l^4_{\tau}L^4}$  with the help of \eqref{prodd1}. For the last term, we use the Sobolev embedding $H^{s_1}\subset L^4$ to get that
$$
\| \Pi_{K} w_{n} \|_{l^4_{\tau}L^4} \lesssim  \| \Pi_{K} w_{n} \|_{l^4_{\tau}H^{s_{1}}}.$$

To get \eqref{prodd3bis}, we just use that
\beq
\label{bisbis1}
\| \Pi_{K}( \Pi_{ \tau^{- {1 \over 2 } }} u_{n} \Pi_{ \tau^{- {1 \over 2 } }} v_{n} \Pi_{K} w_{n} ) \|_{X_\tau^{0, - { 3 \over 8}}}
\lesssim  \|  \Pi_{ \tau^{- {1 \over 2 } }} u_{n} \Pi_{ \tau^{- {1 \over 2 } }} v_{n} \Pi_{K} w_{n}  \|_{X^{0, 0}_\tau}
\eeq
and employ H\"{o}lder's inequality to get
$$
\|  \Pi_{ \tau^{- {1 \over 2 } }} u_{n} \Pi_{ \tau^{- {1 \over 2 } }} v_{n} \Pi_{K} w_{n}  \|_{X_\tau^{0, - { 3 \over 8}}}
\lesssim   \|  \Pi_{ \tau^{- {1 \over 2 } }} u_{n}\|_{l^4_{\tau}L^\infty}    \|  \Pi_{ \tau^{- {1 \over 2 } }} v_{n}\|_{l^4_{\tau}L^\infty} \|\Pi_{K} w_{n}  \|_{l^\infty_{\tau}L^2}.
$$
We then use \eqref{sobbourg} to write
$$
\|\Pi_{K} w_{n}  \|_{l^\infty_{\tau}L^2} \lesssim   \|w_{n}  \|_{X^{0,b}_{\tau}}.
$$
and the  Sobolev embedding $W^{s_1,4}\subset L^\infty$ and  \eqref{prodd1} to obtain
$$
\|  \Pi_{ \tau^{- {1 \over 2 } }} u_{n}\|_{l^4_{\tau}L^\infty}  \lesssim \| u_{n}\|_{X^{s_{1}, {3 \over 8}}_{\tau}}, \quad
\|  \Pi_{ \tau^{- {1 \over 2 } }} v_{n}\|_{l^4_{\tau}L^\infty}  \lesssim \| v_{n}\|_{X^{s_{1}, {3 \over 8}}_{\tau}}.
$$
This concludes the proof of \eqref{prodd3bis}.

For \eqref{prodd3bisbis}, \eqref{prodd3bisbis2},  we just use again \eqref{bisbis1}, and then that
$$ \|  \Pi_{ \tau^{- {1 \over 2 } }} u_{n} \Pi_{ \tau^{- {1 \over 2 } }} v_{n} \Pi_{K} w_{n}  \|_{X^{0, 0}}
  \lesssim  \|\Pi_{ \tau^{- {1 \over 2 } }} u_{n} \|_{l^4_{\tau}L^4}  \|\Pi_{ \tau^{- {1 \over 2 } }} v_{n} \|_{l^4_{\tau}L^4}
   \| \Pi_{K} w_{n}\|_{l^\infty_{\tau}L^\infty}.
$$
To conclude, we use again Lemma \ref{prodd} and the Sobolev embedding $H^{s_{2}} \subset L^\infty$ or Lemma \ref{prodd} and  the Sobolev embeddings $H^{s_{2}} \subset L^\infty$, $H^{s_{1}} \subset L^4$ .
\end{proof}

\begin{rem}
Another version intermediate between \eqref{prodd3ter} and \eqref{prodd3bis} will be also useful. We have that
\beq\label{pourlafin}
\| \Pi_{K}( \Pi_{ \tau^{- {1 \over 2 } }} u_{n} \Pi_{ \tau^{- {1 \over 2 } }} v_{n} \Pi_{K} w_{n} ) \|_{X^{0, - { 3 \over 8}}_{\tau}}\lesssim   (K \tau^{ 1 \over 2})^{ 1 \over 2}  \|u_{n}\|_{X_{\tau}^{0,{3 \over 8}}}
\|v_{n}\|_{X_{\tau}^{s_{1}, {3\over 8}}}\|w_{n}\|_{l^4_{\tau} L^2}
\eeq
with $s_{1}>1/4$. Indeed, we first use \eqref{duald} to get that
$$
\| \Pi_{K}( \Pi_{ \tau^{- {1 \over 2 } }} u_{n} \Pi_{ \tau^{- {1 \over 2 } }} v_{n} \Pi_{K} w_{n} ) \|_{X^{0, - { 3 \over 8}}_{\tau}}
\leq   (K \tau^{ 1 \over 2})^{ 1 \over 2} \| \Pi_{K}( \Pi_{ \tau^{- {1 \over 2 } }} u_{n} \Pi_{ \tau^{- {1 \over 2 } }} v_{n} \Pi_{K} w_{n} ) \|_{l^{4 \over 3}_{\tau}L^{4 \over 3}}
$$
and we employ H\"{o}lder's inequality to get
$$
\| \Pi_{K}( \Pi_{ \tau^{- {1 \over 2 } }} u_{n} \Pi_{ \tau^{- {1 \over 2 } }} v_{n} \Pi_{K} w_{n} ) \|_{l^{4 \over 3}_{\tau}L^{4 \over 3}} \leq
\|   \Pi_{K} w_{n}  \|_{l^4_{\tau}L^2} \|\Pi_{ \tau^{- {1 \over 2 } }} v_{n}\|_{l^4_{\tau}L^\infty} \|\Pi_{ \tau^{- {1 \over 2 } }} u_{n}\|_{l^4_{\tau}L^4}.
$$
We conclude by using the Sobolev embedding $W^{s_{1}, 4}\subset L^\infty $ and \eqref{prodd1}.
\end{rem}

\section{Estimates of the exact solution in discrete Bourgain spaces}
%%%%%%%%%%%%%%%%%%%%%%%%%%%%%%%%%%%%%%%%%%%%%%%%%%%%%%%%%%%%%%%%%%%%%%%%%%%
\label{section4}

In this section, we shall prove that the sequence $u^K(t_{n})$ is an element of $X^{s, b}_{\tau}$ for suitable $s$. It will be convenient to use the following general lemma.

\begin{lemma}\label{lemdisccont}
For any $b \geq 0$,  $b'>1/2 $ and $ s \in \mathbb{R}$,  let us consider a sequence of functions $(u_n(x))_{n\in \mathbb Z}$ of the form  $u_{n}(x)= u(n\tau,x)$. Then
$$
\|u_{n}\|_{X^{s, b}_{\tau}} \lesssim_{b} \|u\|_{X^{s,b+b'}}.
$$
\end{lemma}

\begin{proof}
By setting $f=\e^{-it \partial_{x}^2} u$ and $f_{n}(x)= f(n\tau, x)$, it suffices to prove that
$$
\|f_{n}\|_{H^b_{\tau}L^2} \lesssim \|f\|_{H^{b+b'}L^2},
$$
the extension to general $s$ being straightforward. Since we have by definition that
$$
\widetilde{f_m}(\sigma, k)= \tau \sum_{n\in\mathbb Z} \widetilde f(n\tau, k) \,\e^{in \tau \sigma},
$$
we have by Poisson's summation formula that
$$
\widetilde{f_n}(\sigma, k)= \sum_{m \in \mathbb{Z}} \widetilde f\Bigl( \sigma + { 2 \pi \over \tau} m, k\Bigr).
$$
Therefore,
$$
\langle d_{\tau}(\sigma) \rangle^b \widetilde{f_n}(\sigma, k) = \sum_{m \in \mathbb{Z}} \left\langle d_{\tau}\Bigl(\sigma + { 2 \pi \over \tau} m\Bigr) \right\rangle^b \widetilde f\Bigl( \sigma + { 2 \pi \over \tau} m, k\Bigr),
$$
since $d_{\tau}$ is also a $2\pi/\tau$ periodic function. Since, we always have that $|d_{\tau}(\sigma)| \lesssim \langle \sigma \rangle$, this yields  by Cauchy--Schwarz,
\begin{multline*}
|\langle d_{\tau}(\sigma) \rangle^b \widetilde{f_n}(\sigma, k)|^2 \lesssim \sum_{\mu} { 1 \over \left\langle \sigma +  { 2 \pi \over \tau} \mu \right \rangle^{2b'} }
\sum_{m \in \mathbb{Z}} \left\langle \sigma + { 2 \pi \over \tau} m \right\rangle^{2b + 2b'}\left| \widetilde f\Bigl( \sigma + { 2 \pi \over \tau} m, k\Bigr)\right|^2 \\
\lesssim   \sum_{m \in \mathbb{Z}} \left\langle \sigma + { 2 \pi \over \tau} m \right\rangle^{2b + 2b'}\left| \widetilde f\Bigl( \sigma + { 2 \pi \over \tau} m, k\Bigr)\right|^2
\end{multline*}
since $2b'>1$. By integrating with respect to $\sigma$, we obtain that
$$
\|\langle d_{\tau} \rangle^b \widetilde{f_n}(\cdot, k)\|_{L^2(-\pi/\tau, \pi/\tau)} \lesssim \| \langle \sigma \rangle^{b+ b'} \widetilde f(\cdot, k) \|_{L^2(\mathbb{R})}^2.
$$
We finish the proof by summing over $k$.
\end{proof}

As a consequence of the previous lemma, we obtain the following result.

\begin{proposition}\label{propunifdisc}
Let $u^K$ be the solution of \eqref{nls} and define the sequence  $u^K_{n}(x)= u^K(n\tau + t', x)$. Assume that $u_0 \in H^{s_{0}}$, $s_{0}> 0 $. Then, for every   $s_{1}$,  such that $0 \leq s_{1} <  s_{0}$,  we have that
$$
\sup_{t' \in [0, 4 \tau]} \|\eta(n\tau) u^K_{n}\|_{X^{s_{1}, {3 \over 8}}_{\tau}} \leq C_{T}.
$$
\end{proposition}

\begin{proof}
It suffices to combine Lemma \ref{lemdisccont} and Corollary \ref{corsb} by taking $b'$ arbitrarily close to $1/2$.
\end{proof}

\section{Local error of time discretization}\label{section5}
%%%%%%%%%%%%%%%%%%%%%%%%%%%%%%%%%%%%%%%%%%%%%%%%%%%%%%%%%%

In this section we analyse the time discretization error which is introduced when discretising the twice-filtered Schr\"{o}dinger equation~\eqref{nls} with the scheme~\eqref{scheme}.

Setting
\begin{equation}\label{def:sigma}
\begin{aligned}
\mathcal{A} &=\left \{\epsi = (\epsi_1, \epsi_2,\epsi_3);  \epsi_1, \epsi_2,\epsi_3 \in\left \{ \tau^{-1/2}, K^+\right\} \exists i \neq j : \epsi_i = \epsi_j = \tau^{-1/2} \right\}\\
&=\left \{\left(\tau^{-1/2}, \tau^{-1/2}, \tau^{-1/2}\right),\left(\tau^{-1/2},\tau^{-1/2},K^+\right),\left(\tau^{-1/2}, K^+,\tau^{-1/2}\right),\left(K^+, \tau^{-1/2},\tau^{-1/2}\right) \right\}
\end{aligned}
\end{equation}
allows us to express the filtered Schr\"{o}dinger equation \eqref{nls} as follows
\begin{equation}\label{nlsS}
\begin{aligned}
i \partial_t u^K &= - \partial_{x}^2 u^K
+ \sum_{\epsi \in \mathcal{A}} \Pi_K \left(\Pi_{\epsi_1} u^K \Pi_{\epsi_2} \overline u^K \, \Pi_{\epsi_3} u^K\right)
\end{aligned}
\end{equation}
and Duhamel's formula (with step size $\tau$) takes the form
\begin{equation}\label{duh}
u^K(t_n+\tau) = \e^{i \tau \partial_x^2} u^K(t_n) - i \Pi_K \e^{i \tau \partial_x^2}\sum_{\epsi\in \mathcal{A}} \mathcal{T}_\epsi (u^K)(\tau,t_n),
\end{equation}
where
\begin{equation}\label{def:T}
\mathcal{T}_\epsi (u^K)(\tau,t_n) = \int_0^\tau \e^{-i s\partial_x^2}
 \left(\Pi_{\epsi_1} u^K(t_n+s) \Pi_{\epsi_2} \overline u^K(t_n+s) \, \Pi_{\epsi_3} u^K(t_n+s)\right)\mathrm{d}s.
\end{equation}
Henceforth, we will use the following notation:
$$
V^K_{\epsi_\ell}(s,t) = \e^{is \partial_x^2} \Pi_{\epsi_\ell} u^K(t), \qquad W^K_{\epsi_\ell}(s,t) = \e^{is \partial_x^2} \Pi_{\epsi_\ell} \Pi_K \sum_{\sigma \in \mathcal{A}}  \mathcal{T}_\sigma (u^K)(s,t).
$$
Iterating Duhamel's formula \eqref{duh}, i.e., plugging the expansion
\[
\Pi_{\epsi_\ell}u^K(t_n+s) = \VK{\ell} - i \WK{\ell}
\]
%(which follows by replacing  $\tau$ with $s$ in \eqref{duh})
into \eqref{duh}, yields the representation
\begin{equation}\label{ex2}
\begin{aligned}
u^K(t_n+\tau)
& = \e^{i \tau \partial_{x}^2 }u^K(t_n) - i \Pi_K \e^{i \tau \partial_x^2}\sum_{\epsi \in \mathcal{A}}\int_0^\tau \e^{-i s\partial_x^2}
\Big [\VK{1}\VqK{2}\VK{3}\Big]\dd s \\
 &\qquad - i \Pi_K \e^{i \tau \partial_x^2}\sum_{\epsi \in \mathcal{A}}  E_\epsi(\tau,t_n)
\end{aligned}
\end{equation}
with the remainder
\begin{equation}\label{err1}
E_\epsi(\tau,t_n) =
\int_0^\tau \e^{-i s\partial_x^2}(
E_{\epsi,1}+E_{\epsi,2}+E_{\epsi,3}+E_{\epsi,4}+E_{\epsi,5}+E_{\epsi,6}+E_{\epsi,7})(s,t_n)\dd s,
\end{equation}
defined by
\begin{equation}
\begin{aligned}\label{Ej}
& E_{\epsi,1}(s,t_n) =  i  \VK{1}\WqK{2}\VK{3} \\
& E_{\epsi,2}(s,t_n) =  - i \VK{1}\VqK{2}\WK{3}\\
& E_{\epsi,3}(s,t_n) =  - i \WK{1}\VqK{2}\VK{3} \\
& E_{\epsi,4}(s,t_n) =  \WK{1}\WqK{2}\VK{3}\\
& E_{\epsi,5}(s,t_n) = - \WK{1}\VqK{2}\WK{3}\\
& E_{\epsi,6}(s,t_n) = \VK{1}\WqK{2}\WK{3}\\
& E_{\epsi,7}(s,t_n) = -i \WK{1}\WqK{2}\WK{3}.
\end{aligned}
\end{equation}
%\begin{equation}\label{ex1}
%\begin{aligned}
%u^K(t_n+\tau)
%  &= \e^{i \tau \partial_{x}^2 } u^K(t_n) - i \Pi_K \e^{i \tau \partial_x^2} \sum_{\epsi \in \mathcal{A}} \int_0^\tau \e^{- i s\partial_x^2}
%\Big[\Big(\VK{1} - i \WK{1}\Big)\\
%&\qquad\cdot \Big(\VqK{2} + i \WqK{2}\Big)\Big(\VK{3} - i \WqK{3}\Big)\Big]\dd s\\
%& = \e^{i \tau \partial_{x}^2 }u^K(t_n) - i \Pi_K \e^{i \tau \partial_x^2}\sum_{\epsi \in \mathcal{A}}
%\int_0^\tau \e^{-i s\partial_x^2} \Big [\VK{2}\VqK{2}\VK{3}\Big]\dd s \\
% & - i \Pi_K \e^{i \tau \partial_x^2}\sum_{\epsi \in \mathcal{A}}
%\int_0^\tau \e^{-i s\partial_x^2}(
%E_{\epsi,1}+E_{\epsi,2}+E_{\epsi,3}+E_{\epsi,4}+E_{\epsi,5}+E_{\epsi,6}+E_{\epsi,7})(s,t_n)\mathrm{d}s
%\end{aligned}
%\end{equation}
It remains to analyse the error introduced by  the  time discretization of the integrals in \eqref{ex2}, where the  discretization is carried out in such a way that the {\em dominant} terms in \eqref{ex2}, i.e., the intermediate frequency terms $ \taul < \vert k\vert \leq K $, are solved {\em exactly} while the lower order frequency terms $ |k|\leq \taul$ are approximated in a suitable manner.

%The structure of the introduced error is  given by Lemma \ref{lem:ularge} and Lemma \ref{lem:ubarlarge} below.
\begin{lemma}\label{lem:ularge}
For sufficiently smooth functions $v,w$ it holds that
\begin{align}\label{o1}
\int_0^\tau \e^{-is\partial_{x}^2} \Big[\left( \e^{is \partial_x^2} v\right)\left\vert \e^{is \partial_x^2}  w\right\vert^2\Big]\dd s
=\mathcal{J}^\tau_1(\overline w,v,w) + R_1(\overline w,v,w)
\end{align}
with $\mathcal{J}_1^\tau$ defined in \eqref{J1} and the remainder given by
\begin{multline}\label{R1}
R_1(v_1,v_2,v_3)  = - 2i \int_0^\tau \e^{-i s \partial_x^2} \left[ \left(\e^{i s \partial_x^2}v_2\right) \int_0^s \e^{i (s-s_1)\partial_x^2} \left[
\left(\e^{-i s_1 \partial_x^2} \partial_x^2 v_1\right) \left(\e^{i s_1 \partial_x^2 } v_3\right) \right.\right.\\
\left.\left. + \left( \e^{-i s_1 \partial_x^2} \partial_x v_1\right)\left( \e^{i s_1 \partial_x^2} \partial_x v_3\right)
\right]\dd s_1\right]\dd s.
\end{multline}
\end{lemma}

\begin{proof}
The proof follows two steps. First we will show that in fact
\begin{equation}\label{relationJ1}
\mathcal{J}_1^\tau(\overline w,v,w) =  \int_0^\tau \e^{-is \partial_x^2} \left[\left(\e^{i s \partial_x^2} v\right) \e^{i s \partial_x^2} \vert w \vert^2\right] \dd s.
\end{equation}
The Fourier expansion of the above integral together with the relation
\[
(-k_1+k_2+k_3)^2 - (k_2^2 + (-k_1+k_3)^2) = 2 k_2 (-k_1+k_3)
\]
yields that
\begin{align*}
 \int_0^\tau &\e^{-is \partial_x^2} \left[\left(\e^{i s \partial_x^2} v\right) \e^{i s \partial_x^2} \vert w \vert^2\right] \dd s =
 \sum_{\substack{k_1,k_2,k_3 \in \mathbb{Z}\\k=-k_1+k_2+k_3}}\overline{\hat{w}}_{k_1} \hat{v}_{k_2} \hat{w}_{k_3} \e^{i k x}\int_0^\tau \e^{i s (-k_1+k_2+k_3)^2}
 \e^{- i s (k_2^2 + (-k_1+k_3)^2)} \dd s \\
 & =
  \sum_{\substack{k_1,k_2,k_3 \in \mathbb{Z}\\k=-k_1+k_2+k_3}}\overline{\hat{w}}_{k_1} \hat{v}_{k_2} \hat{w}_{k_3} \e^{i k x}\int_0^\tau \e^{2is k_2 (-k_1+k_3)} \dd s \\
  & = \sum_{\substack{k_1,k_2,k_3 \in \mathbb{Z}\\ k_2 \neq 0,\ k_1 \neq k_3\\k= -k_1+k_2+k_3}}\overline{\hat{w}}_{k_1} \hat v_{k_2} \hat{w}_{k_3} \e^{i k x} \int_0^\tau
 \e^{2is k_2(-k_1+k_3)    }  \dd s +\tau  \hat{v}_0 \vert w\vert^2 +\tau \widehat{(\vert w\vert^2)}_0 (v-\hat{v}_0)\\
& = \sum_{\substack{k_1,k_2,k_3 \in \mathbb{Z}\\ k_2 \neq 0,\ k_1 \neq k_3\\k = -k_1+k_2+k_3}}\overline{\hat{w}}_{k_1}  \hat v_{k_2} \hat{w}_{k_3} \e^{i k x} \frac{
 \e^{2i\tau k_2(-k_1+k_3)}  -1  }{2i  k_2(-k_1+k_3)} +  \tau \hat{v}_0 \vert w \vert^2 +\tau\widehat{(\vert w\vert^2)}_0 (v-\hat{v}_0)\\
 & =  \frac{i}{2} \e^{-i \tau\partial_x^2}\left[ \left(\e^{i \tau \partial_x^2} \partial_x^{-1 }v \right)\e^{i \tau \partial_x^2} \partial_x^{-1} \vert w \vert^2 \right] - \frac{i}{2} \left( \partial_x^{-1} v\right) \partial_x^{-1} \vert w\vert^2  + \tau \hat{v}_0 \vert \red w\black \vert^2 +\tau\widehat{(\vert w\vert^2)}_0 (v-\hat{v}_0)
\end{align*}
which implies \eqref{relationJ1}. Thanks to \eqref{relationJ1} we can furthermore conclude by \eqref{o1} that
\begin{equation}\label{rR1}
\begin{aligned}
R_1(\overline w, v ,w)
%&= \int_0^\tau \e^{-is\partial_{x}^2} \Big[\left( \e^{is \partial_x^2} v\right)\left\vert \e^{is \partial_x^2}  w\right\vert^2\Big]\dd s - \mathcal{J}_1^\tau(\overline w,v,w)
%\\
& = \int_0^\tau \e^{-is\partial_{x}^2} \Big[\left( \e^{is \partial_x^2} v\right)
\Big[\left\vert \e^{is \partial_x^2}  w\right\vert^2 - \e^{i s \partial_x^2} \vert w \vert^2
\Big]
\Big]\dd s.
\end{aligned}
\end{equation}
We note that
\begin{align*}
- 2i  &\int_0^s  \e^{i (s-s_1)\partial_x^2} \left[
\left(\e^{-i s_1\partial_x^2} \partial_x^2 \overline w \right) \left(\e^{i s_1\partial_x^2} w\right) + \left \vert \e^{i s_1 \partial_x^2} \partial_x w\right\vert^2\right]\dd s_1
\\& = -2i \sum_{\ell_1, \ell_2 \in \mathbb{Z}} \overline{\hat w}_{\ell_1} \hat w_{\ell_2}
\e^{i (-\ell_1+\ell_2) x} \e^{-i s (-\ell_1+\ell_2)^2}\int_0^s
\e^{i s_1(-\ell_1+\ell_2)^2} \e^{i s_1 (\ell_1^2 - \ell_2^2)} ( - \ell_1^2 + \ell_1 \ell_2)
\dd s_1
% \\&=  \sum_{\ell_1, \ell_2 \in \mathbb{Z}} \overline{\hat w}_{\ell_1} \hat w_{\ell_2}
%\e^{i (-\ell_1+\ell_2) x} \e^{-i s (-\ell_1+\ell_2)^2}\left(\e^{2 is (-\ell_1) (-\ell_1 +\ell_2)}-1\right)
 \\&=  \sum_{\ell_1, \ell_2 \in \mathbb{Z}} \overline{\hat w}_{\ell_1} \hat w_{\ell_2}
\e^{i (-\ell_1+\ell_2) x} \e^{-i s (-\ell_1+\ell_2)^2}\left(\e^{i s (-\ell_1+\ell_2)^2}\e^{is (\ell_1^2- \ell_2^2)}-1\right)
 \\&=   \sum_{\ell_1, \ell_2 \in \mathbb{Z}} \overline{\hat w}_{\ell_1} \hat w_{\ell_2}
\e^{i (-\ell_1+\ell_2) x}\left(\e^{is (\ell_1^2- \ell_2^2)}-\e^{-i s (-\ell_1+\ell_2)^2}\right) = \left\vert \e^{i s \partial_x^2} w\right\vert^2 - \e^{i s \partial_x^2} \vert w\vert^2.
\end{align*}
Plugging the above relation into \eqref{rR1} yields \eqref{R1}. This concludes the proof.
\end{proof}

\begin{lemma}\label{lem:ubarlarge}
For sufficiently smooth functions $v,w$ it holds that
\begin{equation}\label{o2}
\int_0^\tau \e^{-is\partial_{x}^2} \Big[\left( \e^{-is \partial_x^2} \overline v\right) \left( \e^{is \partial_x^2} w\right)^2 \Big]\dd s
= \mathcal{J}^\tau_2(\overline v,w,w)
    +  R_2(\overline v,w,w)
\end{equation}
with $\mathcal{J}_2^\tau$ defined in \eqref{J2} and the remainder given by
\begin{align}\label{R2}
R_2(v_1,v_2,v_3)
= -2i
\int_0^\tau \e^{-is\partial_{x}^2} \Big[\left( \e^{-is \partial_x^2} v_1\right)
\int_0^s \e^{i (s-s_1)\partial_x^2} \left(\e^{i s_1 \partial_x^2} \partial_x v_2\right)\left(\e^{i s_1 \partial_x^2} \partial_x v_3\right)\dd s_1
 \Big]\dd s.
\end{align}
\end{lemma}
\begin{proof}
Again we prove the assertion in two steps. First we show that in fact
\begin{equation}\label{J2rel}
\mathcal{J}^\tau_2(\overline v,w,w)=\int_0^\tau  \e^{-i s \partial_x^2} \Big[\left(\e^{-is \partial_x^2} \overline{v} \right)\left(\e^{i s \partial_x^2} w^2\right)\Big] \dd s .
\end{equation}
The above assertion follows by Fourier expansion of the integral together with the relation
\[
 (-k_1+k_2+k_3)^2 + k_1^2 - (k_2+k_3)^2  = - 2 k_1 (-k_1+k_2+k_3)
\]
which implies that
\begin{equation}
\begin{aligned}
\int_0^\tau&  \e^{-i s \partial_x^2} \Big[\left(\e^{-is \partial_x^2} \overline{v} \right)\left(\e^{i s \partial_x^2} w^2\right)\Big] \dd s
 =  \sum_{\substack{k_1,k_2,k_3 \in \mathbb{Z}\\k = - k_1+k_2+k_3}}\overline{\hat{v}}_{k_1} \hat w_{k_2} \hat{w}_{k_3} \e^{i k x} \int_0^\tau
 \e^{i s (-k_1+k_2+k_3)^2 } \e^{i s (k_1^2 - (k_2+k_3)^2)}\dd s
 \\&  = \sum_{\substack{k_1,k_2,k_3 \in \mathbb{Z}\\k = - k_1+k_2+k_3}}\overline{\hat{v}}_{k_1} \hat w_{k_2} \hat{w}_{k_3} \e^{i k x} \int_0^\tau
 \e^{- 2 is k_1 k  } \dd s
 = \sum_{\substack{k_1,k_2,k_3 \in \mathbb{Z}\\k = - k_1+k_2+k_3}}\overline{\hat{v}}_{k_1} \hat w_{k_2} \hat{w}_{k_3} \e^{i k x} \frac{
 \e^{-2i  k_1 k  \tau} - 1}{-2i k_1 k}\\
 &=  \frac{i}{2} \e^{-i \tau \partial_x^2}\partial_x^{-1} \Big[\left(\e^{-i\tau \partial_x^2} \partial_x^{-1}\overline{v} \right)\left(\e^{i \tau \partial_x^2} w^2\right)\Big] - \frac{i}{2} \partial_x^{-1} \left(w^2 \partial_x^{-1} \overline v\right) + \tau \widehat{\left(\overline v w^2\right)}_0 + \tau \overline{\hat{v}_0} \big(w^2 -  \widehat{(w^2)}_0 \big).
\end{aligned}
\end{equation}
Thanks to \eqref{J2rel} we can furthermore conclude by \eqref{o2} that
\begin{equation}\label{rR2}
\begin{aligned}
 R_2(\overline v, w,w)
 %& = \int_0^\tau \e^{-is\partial_{x}^2} \Big[\left( \e^{-is \partial_x^2} \overline v\right) \left( \e^{is \partial_x^2} w\right)^2 \Big]\dd s
%- \mathcal{J}^\tau_2(\overline v,w,w) \\
& =  \int_0^\tau \e^{-is\partial_{x}^2} \Big[\left( \e^{-is \partial_x^2} \overline v\right) \left[
\left( \e^{is \partial_x^2} w\right)^2
- \e^{is\partial_x^2} w^2
\right]\Big]\dd s.
\end{aligned}
\end{equation}
We note that
\begin{align*}
-2i& \int_0^s  \e^{i (s-s_1) \partial_x^2} \left(\e^{is_1 \partial_x^2} \partial_x w\right)^2 \dd s_1  \\&=2i \sum_{\ell_1, \ell_2 \in \mathbb{Z}} \hat w_{\ell_1} \hat w_{\ell_2} \e^{i (\ell_1+\ell_2) x} \e^{- i s (\ell_1+\ell_2)^2} \int_0^s \e^{ i s_1 (\ell_1+\ell_2)^2} \e^{- i s_1 (\ell_1^2+\ell_2^2)} \ell_1 \ell_2 \dd s_1\\
 & = 2i \sum_{\ell_1, \ell_2 \in \mathbb{Z}} \hat w_{\ell_1} \hat w_{\ell_2} \e^{i (\ell_1+\ell_2) x} \e^{- i s (\ell_1+\ell_2)^2} \int_0^s \e^{ 2 i s_1 \ell_1 \ell_2 }\ell_1 \ell_2\dd s_1\\
  & =  \sum_{\ell_1, \ell_2 \in \mathbb{Z}} \hat w_{\ell_1} \hat w_{\ell_2} \e^{i (\ell_1+\ell_2) x} \e^{- i s (\ell_1+\ell_2)^2} \left( \e^{  i s( \ell_1 +\ell_2)^2} \e^{-i s(\ell_1^2+\ell_2^2)} -1\right) = \left( \e^{is \partial_x^2} w\right)^2 -  \e^{is\partial_x^2} w^2.
\end{align*}
Plugging the above relation into \eqref{rR2} proves the assertion.
\end{proof}

\begin{lemma}[Local error]\label{lem:localerror}
The local error
\[
 \mathcal{E}(\tau,t_n) := u^K(t_n+\tau) - \PhiKtau(u^K(t_n))
 \]
  of the time discretization scheme  \eqref{scheme} applied to the filtered Schr\"{o}dinger equation \eqref{nls} reads
   \begin{equation}
\begin{aligned}\label{localerror}
  \mathcal{E}(\tau,t_n) &= - 2i \Pi_K \e^{i \tau \partial_x^2} R_1( \Pi_{\taul} \overline u^K(t_n),  \Pi_{K^+} u^K(t_n), \Pi_{\taul} u^K(t_n))\\
 & \quad - i \Pi_K \e^{i \tau \partial_x^2} R_2( \Pi_{K} \overline u^K(t_n),  \Pi_{\taul} u^K(t_n), \Pi_{\taul} u^K(t_n))\\
& \quad - i \Pi_K \e^{i \tau \partial_x^2}\sum_{\epsi \in \mathcal{A}}
E_\epsi(\tau,t_n),
\end{aligned}
\end{equation}
where $E_\epsi(\tau,t_n)$ is defined in \eqref{err1} and the remainders  $R_1$ and $R_2$ are given in \eqref{R1} and \eqref{R2}, respectively.
\end{lemma}
\begin{proof}
The assertion follows by the expansion of the exact solution $u^K(t_n+\tau)$ given in \eqref{ex2} together with Lemmas \ref{lem:ularge} and \ref{lem:ubarlarge}.

More precisely, employing Lemma \ref{lem:ularge} to approximate the integral arising for $\epsi_1 = K^+$ or $\epsi_3 = K^+$ in \eqref{ex2} and Lemma \ref{lem:ubarlarge} to approximate the integrals arising for $\epsi_2 = K^+$ or $\epsi_1 = \epsi_2 = \epsi_3 = \tau^{-1/2}$ in \eqref{ex2} yields that
\begin{equation}\label{sol:exp}
\begin{aligned}
u^K(t_n+\tau)
& = \e^{i \tau \partial_{x}^2 }u^K(t_n)  \\
& \quad-2 i \Pi_K \e^{i\tau \partial_x^2} \mathcal{J}_1^\tau\left(\Pi_{\taul} \overline u^K(t_n), \Pi_{K^+} u^K(t_n), \Pi_{\taul} u^K(t_n) \right)\\
& \quad - i\Pi_K \e^{i\tau \partial_x^2}
 \mathcal{J}_2^\tau\left(\Pi_{K} \overline u^K(t_n), \Pi_{\taul} u^K(t_n), \Pi_{\taul} u^K(t_n) \right)\\
 & \quad - 2i \Pi_K \e^{i \tau \partial_x^2} R_1( \Pi_{\taul} \overline u^K(t_n),  \Pi_{K^+} u^K(t_n), \Pi_{\taul} u^K(t_n))\\
 & \quad - i \Pi_K \e^{i \tau \partial_x^2} R_2( \Pi_{K} \overline u^K(t_n),  \Pi_{\taul} u^K(t_n), \Pi_{\taul} u^K(t_n))\\
& \quad - i \Pi_K \e^{i \tau \partial_x^2}\sum_{\epsi \in \mathcal{A}}
E_\epsi(\tau,t_n).
\end{aligned}
\end{equation}
The assertion thus  follows by taking the difference of the  expansion of the exact solution given in \eqref{sol:exp}  and the numerical flow defined in \eqref{scheme}.
\end{proof}

%%%%%%%%%%%%%%%%%%%%%%%%%%%%%%%%%%%%%%%
\section{Global error analysis}\label{section6}
%%%%%%%%%%%%%%%%%%%%%%%%%%%%%%%%%%%%%%%

Let $e^{n+1} = u^K(t_{n+1}) - \unum{n+1}$ denote the time discretization error, i.e., the difference between the numerical solution $\unum{n+1} = \PhiKtau(\un)$ defined in \eqref{scheme} and the exact solution of the filtered Schr\"{o}dinger equation \eqref{nls}. Inserting a zero in terms of $\pm \PhiKtau(u^K(t_n))$, i.e., using that $$e^{n+1}  = u^K(t_{n+1}) - \PhiKtau(u^K(t_n)) +  \PhiKtau(u^K(t_n)) - \PhiKtau\un$$ we obtain by the definition of the numerical flow $\PhiKtau$ in \eqref{scheme} that
\begin{equation}\label{global}
\begin{aligned}
e^{n+1}  = \e^{i \tau\partial_x^2} e^n
&{}- 2i \Pi_K \e^{i\tau \partial_x^2} \left[\mathcal{J}_1^\tau\left(\Pi_{\taul} \overline u^K(t_n), \Pi_{K^+} u^K(t_n), \Pi_{\taul} u^K(t_n) \right) \right.\\
&\qquad\qquad\qquad\qquad\qquad\qquad\qquad - \left.\mathcal{J}_1^\tau\left(\Pi_{\taul} \unq, \Pi_{K^+} \un, \Pi_{\taul} \un \right)\right]\\
&{} - i\Pi_K \e^{i\tau \partial_x^2} \left[\mathcal{J}_2^\tau\left(\Pi_{K} \overline u^K(t_n), \Pi_{\taul} u^K(t_n), \Pi_{\taul} u^K(t_n) \right) \right.\\
&\qquad\qquad\qquad\qquad\qquad\qquad\qquad - \left.\mathcal{J}_2^\tau\left(\Pi_{K} \unq, \Pi_{\taul} \un, \Pi_{\taul} \un\right) \right]\\
&{} + \mathcal{E}(\tau,t_n),
\end{aligned}
\end{equation}
where $\mathcal{J}^\tau_1$ and $\mathcal{J}^\tau_2$ are defined in \eqref{J1} and \eqref{J2} and the local error \eqref{localerror} is given in Lemma \ref{lem:localerror}.

By solving the above recursion, we get that for $0 \leq n \leq N_{1}=\lfloor {T_{1}\over \tau}\rfloor$ with $T_{1} \leq T$, the global error
$e^n$ satisfies
\begin{equation}
\label{erroreq}
e^{n} = \tau \eta(t_n)  \sum_{k=0}^{n-1}  \e^{i(n-k) \tau \partial_{x}^2} \eta\left({k \tau \over T_{1}}\right) \Pi_{K} G_k + \mathcal{R}_{1, n} +  \mathcal{R}_{2, n},
\end{equation}
where we have set
\begin{align}
\label{defF}
G_{n} &= \tfrac{-2i}{\tau} \Bigl[\mathcal{J}_1^\tau\left(\Pi_{\taul} \overline u^K(t_n), \Pi_{K^+} u^K(t_n), \Pi_{\taul} u^K(t_n) \right)
- \mathcal{J}_1^\tau\left(\Pi_{\taul} \unq, \Pi_{K^+} \un, \Pi_{\taul} \un \right)\Bigr] \nonumber \\
&\quad - \tfrac{i}{\tau}  \Bigl[\mathcal{J}_2^\tau\left(\Pi_{K} \overline u^K(t_n), \Pi_{\taul} u^K(t_n), \Pi_{\taul} u^K(t_n) \right)
 - \mathcal{J}_2^\tau\left(\Pi_{K} \unq, \Pi_{\taul} \un, \Pi_{\taul} \un\right)\Bigr],
\end{align}
and the remainders
\begin{equation*}
\mathcal{R}_{i, n} = \tau \eta(t_n)\sum_{k=0}^{n-1}  \e^{i(n-k) \tau \partial_{x}^2}   \eta(t_k) \Pi_{K}\mathcal{F}_{i}(t_{k}), \qquad i=1,2,
\end{equation*}
with
\begin{align}\red
\label{defG1}
\mathcal{F}_{1}(t_{n})& = { 1 \over \tau } \left(
 - 2i \Pi_K  R_1( \Pi_{\taul} \overline u^K(t_n),  \Pi_{K^+} u^K(t_n), \Pi_{\taul} u^K(t_n)) \right. \nonumber\\
&\qquad\qquad  \left.{}- i \Pi_K   R_2( \Pi_{K} \overline u^K(t_n),  \Pi_{\taul} u^K(t_n), \Pi_{\taul} u^K(t_n))\right),\\
\label{defG2}\red
\mathcal{F}_{2}(t_{n}) &= -{i \over \tau} \Pi_{K} \sum_{\epsi \in \mathcal{A}} E_\epsi(\tau,t_n).
\end{align}
Note that $E_{\epsi}$ is defined in \eqref{err1} and  $R_{1}$, $R_{2}$ in \eqref{R1} and \eqref{R2}. We have introduced the truncation function $\eta$ in order to work with global Bourgain spaces. As before we will assume that $\un$ and $u^K$ are globally defined though they coincide with the actual solutions of the scheme and the PDE on a finite interval of time. We will choose $T_{1}$ sufficiently small later.

We shall first  estimate $\mathcal{R}_{1, n}$, which gives the dominant contribution to the error.

\begin{lemma} \label{lemmaR1n}
Let $s_0\in(0,1]$ and $b \in (1/2, 5/8)$. For $s_{0} >0$ we, have the estimate
\beq\label{lem61-a}
\| \mathcal{R}_{1, n} \|_{X^{0, b}_{\tau}} \leq  C_{T} K\tau^{1 \over 2 }  \tau^{ (s_{0})_{-}}.
\eeq
Moreover, if $s_{0}>1/4$, we have
\beq\label{lem61-b}
\| \mathcal{R}_{1, n} \|_{X^{0, b}_{\tau}} \leq  C_{T} (K\tau^{1 \over 2 })^{1 \over 2}\tau^{ s_{0}-  ( { 1 \over 8})_{+}},
\eeq
and if  $s_{0} >{ 1 \over 2}$, we have
\beq\label{lem61-c}
\| \mathcal{R}_{1, n} \|_{X^{0, b}_{\tau}} \leq  C_{T}\tau^{ s_{0}-  ( { 1 \over 8})_{+}}.
\eeq
\end{lemma}

\begin{proof}
By using again \eqref{bourg3} and \eqref{bourg4}, we get that
 $$ \| \mathcal{R}_{1, n} \|_{X^{0, b}_{\tau}} \lesssim \| \red \mathcal{F}_{1}(t_n)\black \|_{X^{0, {-{3\over 8}}}_{\tau}}.$$
  By using \eqref{defG1} this amounts to estimate  %and \eqref{R1}, \eqref{R2}
  \begin{align*}
  &  I_{1}= {1 \over \tau } \|R_1( \Pi_{\taul} \overline u^K(t_n),  \Pi_{K^+} u^K(t_n), \Pi_{\taul} u^K(t_n))\|_{X^{0, - {3 \over 8}}_{\tau}}, \\
  &  I_{2}={ 1 \over \tau} \|R_2( \Pi_{K} \overline u^K(t_n),  \Pi_{\taul} u^K(t_n), \Pi_{\taul} u^K(t_n))\|_{X^{0, - {3 \over 8}}_{\tau}}.
  \end{align*}
  We first prove \eqref{lem61-a}.
  We start with the estimate of  $I_{2}$.
  We use \eqref{R2}, \eqref{prodd3} and Remark \ref{remshift}
  to  obtain that
 $$
 I_{2} \lesssim K\tau^{1 \over 2} \| u^K(t_{n})\|_{X_{\tau}^{0, {3 \over 8}}} \| \tau^{1 \over 2} \partial_{x}
  \Pi_{\tau^{- {1/2}}} u^K(t_{n})\|_{X_{\tau}^{0, {3 \over 8}}}^2.
$$
By using  Proposition  \ref{propunifdisc}, this yields
$$ I_{2} \lesssim K\tau^{1 \over 2} \| u^K\|_{X^{s_{0}, b}} \| \tau^{1 \over 2} \partial_{x}
  \Pi_{\tau^{- {1/2}}} u^K\|_{X^{s_{2}, b}}^2
$$
  with $s_{0}>s_{2}>0$, $s_{2}$ arbitrarily small
   and $b\in (7/8, 1)$.  Consequently, from the frequency localization, we find that
$$
 I_{2} \lesssim \tau^{ s_{0}-s_{2}} (K\tau^{1 \over 2})  \| u^K\|_{X^{s_{0}, b}} \|   \Pi_{\tau^{- {1/2}}} u^K\|_{X^{s_{0}, b}}^2
  \leq C_{T} \tau^{ s_{0}- s_{2}}\,(K\tau^{1 \over 2}),
$$
{where we use Corollary \ref{corsb} for the last estimate}.

  It remains to estimate $I_{1}$.  By using the definition \eqref{R1} and the same arguments, we get that
$$ I_{1} \lesssim   K\tau^{1 \over 2} \| (\tau^{1 \over 2} \partial_{x})^2  \Pi_{ \tau^{- {1/2}}} u \|_{X^{s_{2}, b}}
   \| \Pi_{K^{+}}u^K\|_{X^{s_{2}, b}}  \, \|u^K\|_{X^{s_{2}, b}} +
   K\tau^{1 \over 2} \| u^K\|_{X^{s_{0}, b}} \| \tau^{1 \over 2} \partial_{x}
  \Pi_{\tau^{- {1/2}}} u^K\|_{X^{s_{2}, b}}^2
$$
 again with $s_{0}>s_{2}>0, $ $s_{2}$ arbitrarily small.  The second term is similar as before. For the first term, by using the frequency localization, in particular the fact that on the support of $\Pi_{K^{+}}$,  $ \tau^{1 \over 2}|\xi | \geq 1$,
  we then obtain that
$$   \| (\tau^{1 \over 2} \partial_{x})^2  \Pi_{ \tau^{- {1/2}}} u \|_{X^{s_{2}, b}}
  \lesssim \tau^{ (s_{0}- s_{2})/2}   \| u \|_{X^{s_{0}, b}}, \quad  \| \Pi_{K^{+}}u^K\|_{X^{s_{2}, b}}
   \lesssim   \tau^{ (s_{0}- s_{2})/2} \| u^K\|_{X^{s_{0}, b}}.
$$
This also yields
$$
I_{1}\leq C_{T} \tau^{ s_{0}- s_{2}}\,(K\tau^{1 \over 2}),
$$
which concludes the proof of \eqref{lem61-a}.

To prove \eqref{lem61-b}, we follow the same lines, but we use \eqref{prodd3ter} instead of \eqref{prodd3} since $s_{0}>1/4$.  This yields
  \begin{multline*}
  \| \mathcal{R}_{1, n} \|_{X^{0, b}_{\tau}}  \lesssim
 ( K\tau^{1 \over 2})^{1 \over 2} \| u^K (t_{n})\|_{l^4_{\tau}H^{({1 \over 4})_{+} }}  \| \tau^{1 \over 2} \partial_{x}
  \Pi_{\tau^{- {1 \over 2} } } u^K(t_{n})\|_{X_{\tau}^{0, {3\over 8} }}^2
   \\+  (K\tau^{1 \over 2})^{1 \over 2} \| (\tau^{1 \over 2} \partial_{x})^2  \Pi_{ \tau^{- {1 \over 2}}} u^K (t_{n}) \|_{X_{\tau}^{0, {3 \over 8 } } }
\| \Pi_{K^+}(u^K(t_{n})\|_{l^4_{\tau}H^{({1 \over 4})_{+} }}    \, \| \Pi_{ \tau^{- {1 \over 2}}} u^K (t_{n})\|_{X_{\tau}^{0, {3 \over 8}}}.
  \end{multline*}
By using again the same estimates as above,  it  thus only remains to estimate
$\| u^K (t_{n})\|_{l^4_{\tau}H^{({1 \over 4})_{+} }}$ and $\| \Pi_{K^+}(u^K(t_{n})\|_{l^4_{\tau}H^{({1 \over 4})_{+} }}$.
We can just use that
$$ \| u^K (t_{n})\|_{l^4_{\tau}H^{({1 \over 4})_{+} }}
 \lesssim T^{ 1 \over 4} \|u^K\|_{L^\infty_{T}H^{s_{0}}} \lesssim T^{ 1 \over 4} \|u^K\|_{X^{s_{0}, b}},\qquad b>1/2
 $$
 and, by frequency localization for $|\xi| \geq \tau^{- { 1 \over 2}}$, that
\beq
\label{bisbis2} \| \Pi_{K^+}(u^K(t_{n})\|_{l^4_{\tau}H^{({1 \over 4})_{+} }}  \lesssim_{T}  \tau^{{ 1 \over 2} ({s_{0}- ({1\over 4})_{+}})} \|u^K\|_{L^\infty_{T}H^{s_{0}}}.
\eeq
This yields \eqref{lem61-b}.

Finally to get \eqref{lem61-c}, we follow the same lines but we now use \eqref{prodd3bisbis} and  \eqref{prodd3bisbis2}  . This yields
\begin{multline*}  \| \mathcal{R}_{1, n} \|_{X^{0, b}_{\tau}} \lesssim   \| \tau^{1 \over 2} \partial_{x}
  \Pi_{\tau^{- {1 \over 2} } } u^K(t_{n})\|_{X_{\tau}^{0, {3\over 8} }}^2  \| u^K (t_{n})\|_{l^\infty_{\tau}H^{({1 \over 2})_{+} }}
   \\+  \| (\tau^{1 \over 2} \partial_{x})^2  \Pi_{ \tau^{- {1 \over 2}}} u^K (t_{n}) \|_{X_{\tau}^{0, {3 \over 8 } } }
   \, \| \Pi_{K^+}u^K (t_{n})\|_{l^4_{\tau}H^{({1 \over 4})_{+} }}  \|  u^K (t_{n}) \|_{l^\infty_{\tau}H^{({1 \over 2})_{+} }}.
 \end{multline*}
We then use the same estimates, in particular \eqref{bisbis2} and the fact that
$$
\| u^K (t_{n})\|_{l^\infty_{\tau}H^{({1 \over 2})_{+} }} \leq \|u^K\|_{L^\infty_{T}H^{({1 \over 2})_{+} }}.
$$
This ends the proof.
\end{proof}

We shall next  estimate $\mathcal{R}_{2,n}$.

\begin{lemma}
\label{lemmaR2n}
For $s_{0}> 0$ and  $b \in (1/2, 5/8)$, we have the estimate
\begin{equation}
\label{R2n1} \| \mathcal{R}_{2, n} \|_{X^{0, b}_{\tau}} \leq C_{T} \left(\tau^{1 \over 4}  (K \tau^{1 \over 2})^{2} + \tau^{1 \over 2} (K\tau^{1 \over 2})^3
 + \tau^{3 \over 2} (K \tau^{1 \over 2})^{4}\right).
 \end{equation}
 Moreover if $s_{0}>{ 1 \over 4}$, we have
 \begin{equation}
 \label{R2n2}   \| \mathcal{R}_{2, n} \|_{X^{0, b}_{\tau}} \leq C_{T}(K \tau^{1 \over 2})^{1 \over 2} \tau^{ 5 \over 8}
 \end{equation}
 and if $s_{0}>1/2$, we have
 \begin{equation}
 \label{R2n3}   \| \mathcal{R}_{2, n} \|_{X^{0, b}_{\tau}}  \leq C_{T} \tau.
 \end{equation}
\end{lemma}

\begin{proof}
We first use \eqref{bourg4} and  Remark \ref{remshift} to estimate
$$
\| \mathcal{R}_{2, n} \|_{X^{0, b}_{\tau}} \lesssim \|  \red \mathcal{F}_{2}(t_n) \|_{X^{0, {-{3\over 8}}}_{\tau}}.
$$
 Next, by using \eqref{err1} and the product estimate \eqref{prodd3}, we get
 \begin{multline*}
 \| \mathcal{F}_{2, n} \|_{X^{0, {-{3\over 8}}}_{\tau}}
  \lesssim K \tau^{1 \over 2} \sum_{\kappa \in \mathcal{A}}   \left(
   \| u^K(t_{n}) \|_{X^{0, {3 \over 8}}}^2   \|\Tkappa (u^K(t_{n})) \|_{X^{0, {3 \over 8}}_{\tau}}
 \right.  \\
 \left. {}+   \| u^K(t_{n}) \|_{X^{0, {3 \over 8}}_{\tau}}   \|\Tkappa (u^K(t_{n})) \|_{X^{0, {3 \over 8}}_{\tau}}^2 +
      \|\Tkappa (u^K(t_{n})) \|_{X^{0, {3 \over 8}}_{\tau}}^3\right).
 \end{multline*}
  Next, by using \eqref{def:T}, we get that
 $$
 \|\Tkappa (u^K(t_{n})) \|_{X^{0, {3 \over 8}}_{\tau}} \lesssim \tau \sup_{t' \in [0, \tau]}
  \sum_{\epsi \in \mathcal{A}}\left\| \Pi_{\epsi_{1}}
  u^K(t_{n}+ t') \Pi_{\epsi_{2}} \overline{u}^K(t_{n}+ t')   \Pi_{\epsi_{3}} \overline{u}^K(t_{n}+ t')  \right\|_{X^{0, {3 \over 8}}_{\tau}}.
  $$
By using \eqref{embdisc1}, we thus obtain that
$$ \|\Tkappa (u^K(t_{n})) \|_{X^{0, {3 \over 8}}_{\tau}} \lesssim \tau^{1 \over 4} \sup_{t' \in [0, \tau]}
  \sum_{\epsi \in \mathcal{A}}\left\| \Pi_{\epsi_{1}}
  u^K(t_{n}+ t') \Pi_{\epsi_{2}} \overline{u}^K(t_{n}+ t')   \Pi_{\epsi_{3}} \overline{u}^K(t_{n}+ t')  \right\|_{X^{0, {-{3 \over 8}}}_{\tau}}.$$
  Consequently, by using again the product estimate \eqref{prodd3}, we find that
 $$   \|\Tkappa (u^K(t_{n})) \|_{X^{0, {3 \over 8}}_{\tau}} \lesssim \tau^{1 \over 4} (K \tau^{1 \over 2})
    \sup_{t' \in [0, \tau]} \|u^K(t_{n}+ t')\|_{X^{0, {3 \over 8}}_{\tau}}^3.$$
  Then, by using Proposition \ref{propunifdisc}, we get
  $$  \|\Tkappa (u^K(t_{n})) \|_{X^{0, {3 \over 8}}_{\tau}} \lesssim \tau^{1 \over 4} (K \tau^{1 \over 2}) C_{T}$$
  and hence
  $$   \| \mathcal{R}_{2, n} \|_{X^{0, b}_{\tau}} \lesssim K \tau^{1 \over 2} \left( \tau^{1 \over 4}  (K \tau^{1 \over 2})
   +   (\tau^{1 \over 4}  (K \tau^{1 \over 2}))^2 +  (\tau^{1 \over 4}  (K \tau^{1 \over 2}))^3 \right) C_{T}. $$
  This proves \eqref{R2n1}.

 To get \eqref{R2n2}, we use \eqref{duald} and \eqref{prodd3ter} to get this yields
 \begin{multline*}
 \| \mathcal{R}_{2, n} \|_{X^{0, b}_{\tau} } \lesssim \|  \red \mathcal{F}_{2}(t_n) \black \|_{X^{0, -{3\over 8} }_{\tau} }
 \\ \lesssim   (K \tau^{1 \over 2})^{1 \over 2 } \sum_{\kappa \in \mathcal{A}} \Bigl(
   \| u^K(t_{n}) \|_{X^{0, {3 \over 8} }_{\tau}}^2   \|\Tkappa (u^K(t_{n})) \|_{l^4_{\tau} H^{({1 \over 4})_{+}}}
    +  \| u^K(t_{n}) \|_{X^{0, {3 \over 8} }_{\tau}}  \| u^K(t_{n}) \|_{l^4_{\tau} H^{({1 \over 4})_{+}}}
     \|\Tkappa (u^K(t_{n})) \|_{X^{0, {3 \over 8} }_{\tau}} \\
      +  \|\Tkappa (u^K(t_{n})) \|_{X^{0, {3 \over 8} }_{\tau}}^2 \| u^K(t_{n}) \|_{l^4_{\tau} H^{({1 \over 4})_{+}}}
      + \|\Tkappa (u^K(t_{n})) \|_{X^{0, {3 \over 8} }_{\tau}}\| u^K(t_{n}) \|_{X^{0, {3 \over 8} }_{\tau}}
      \|\Tkappa (u^K(t_{n})) \|_{l^4_{\tau} H^{({1 \over 4})_{+}}} \\
       +  \|\Tkappa (u^K(t_{n})) \|_{X^{0, {3 \over 8} }_{\tau}}^2 \| \Tkappa (u^K(t_{n})) \|_{l^4_{\tau} H^{({1 \over 4})_{+}}}\Bigr).
\end{multline*}
We can use again \eqref{embdisc1} and  the trivial estimate
$$ \| u^K(t_{n}) \|_{l^4_{\tau} H^{({1 \over 4})_{+}}} \lesssim T^{1 \over 4}  \| u^K \|_{L^\infty H^{({1 \over 4})_{+}}}$$
so that it only remains to estimate  $\|\Tkappa (u^K(t_{n})) \|_{X^{0, {3 \over 8} }_{\tau}}$ and
$\|\Tkappa (u^K(t_{n})) \|_{l^4_{\tau} H^{({1 \over 4})_{+}}}$.
For the first one, we use again \eqref{embdisc1} to write
$$
\|\Tkappa (u^K(t_{n})) \|_{X^{0, {3 \over 8} }_{\tau}} \lesssim  \tau^{5 \over 8} \sup_{t' \in [0, \tau]}
\sum_{\epsi \in \mathcal{A}}\left\| \Pi_{\epsi_{1}}
u^K(t_{n}+ t') \Pi_{\epsi_{2}} \overline{u}^K(t_{n}+ t')   \Pi_{\epsi_{3}} u^K(t_{n}+ t')  \right\|_{X^{0, 0}_{\tau}}.
$$
Next, from H\"{o}lder's inequality and the Sobolev embedding $W^{ ({ 1 \over 4} )_{+}, 4} \subset L^\infty$, we get that
$$
\|\Tkappa (u^K(t_{n})) \|_{X^{0, {3 \over 8} }_{\tau}} \lesssim  \tau^{5 \over 8} \sup_{t' \in [0, \tau]}
\lesssim  \| \Pi_{\tau^{- {1 \over 2}}} u^K(t_{n}+ t') \|_{l^4_{\tau}W^{ ({ 1 \over 4} )_{+}, 4}}^2 \| \Pi_{K} u^K(t_{n} + t') \|_{l^\infty_{\tau} L^2}
$$
and hence by using again \eqref{prodd1}, Proposition \ref{propunifdisc} and \eqref{bshift} we get that
for $s_{0}>1/4$,
$$
\|\Tkappa (u^K(t_{n})) \|_{X^{0, {3 \over 8} }_{\tau}} \lesssim  \tau^{5 \over 8} \|u^K\|^3_{X^{s_{0},b}} \lesssim C_{T}  \tau^{5 \over 8}.
$$
Next, we estimate $\|\Tkappa (u^K(t_{n})) \|_{l^4_{\tau} H^{({1 \over 4})_{+}}}$. We begin with
\begin{equation}
\label{prodhigh2}
\|\Tkappa (u^K(t_{n})) \|_{l^4_{\tau} H^{({1 \over 4})_{+}}} \lesssim \tau \sup_{t' \in [0, \tau]} \sum_{\epsi \in \mathcal{A}}\left\| \Pi_{\epsi_{1}}
u^K(t_{n}+ t') \Pi_{\epsi_{2}} \overline{u}^K(t_{n}+ t')   \Pi_{\epsi_{3}} u^K(t_{n}+ t')  \right\|_{l^4_{\tau} H^{({1 \over 4})_{+}}}.
\end{equation}
{Next, we observe  that for all sequences $(u_{n})$, $(v_{n})$, $(w_{n})$, and $s>0$ we have
\begin{equation} 
\label{prodhigh}
\| \Pi_{K^+} u_{n}   \Pi_{\tau^{- {1 \over 2} }} v_{n}  \Pi_{\tau^{- {1 \over 2} }}w_{n} \|_{H^s}
 \lesssim \|  \Pi_{K^+}u_{n}\|_{H^s} \|   \Pi_{\tau^{- {1 \over 2} }} v_{n}\|_{L^\infty} \|   \Pi_{\tau^{- {1 \over 2} }} w_{n}\|_{L^\infty}.
\end{equation}
Indeed, by using the generalized Leibniz rule \eqref{Leibniz}, we have that
\begin{multline*}  \| \Pi_{K^+} u_{n}   \Pi_{\tau^{- {1 \over 2} }} v_{n}  \Pi_{\tau^{- {1 \over 2} }}w_{n} \|_{H^s}
 \lesssim  \|  \Pi_{K^+}u_{n}\|_{H^s} \|   \Pi_{\tau^{- {1 \over 2} }} v_{n}\|_{L^\infty} \|   \Pi_{\tau^{- {1 \over 2} }} w_{n}\|_{L^\infty}
  \\+  \|  \Pi_{K^+}u_{n}\|_{L^2} \|  \langle \partial_{x} \rangle^s \Pi_{\tau^{- {1 \over 2} }} v_{n}\|_{L^{\infty}} \|   \Pi_{\tau^{- {1 \over 2} }} w_{n}\|_{L^\infty}
   +  \|  \Pi_{K^+}u_{n}\|_{L^2} \|   \Pi_{\tau^{- {1 \over 2} }} v_{n}\|_{L^\infty} \| \langle \partial_{x} \rangle^s  \Pi_{\tau^{- {1 \over 2} }} w_{n}\|_{L^\infty}.\end{multline*}
By frequency localization, we observe that
$$   \|  \Pi_{K^+}u_{n}\|_{L^2} \lesssim \tau^{s \over 2} \|u_{n}\|_{H^s}, \quad
 \|  \langle \partial_{x} \rangle^s \Pi_{\tau^{- {1 \over 2} }} v_{n}\|_{L^{\infty}}
  \lesssim { 1 \over \tau^{s \over 2}} \|v_{n}\|_{L^\infty}, \quad
  \|  \langle \partial_{x} \rangle^s \Pi_{\tau^{- {1 \over 2} }} w_{n}\|_{L^{\infty}}
  \lesssim { 1 \over \tau^{s \over 2}} \|w_{n}\|_{L^\infty} $$
  and hence
  \eqref{prodhigh} follows.
  }
 We thus deduce from \eqref{prodhigh2} and \eqref{prodhigh} that
\begin{align*}  \|\Tkappa (u^K(t_{n})) \|_{l^4_{\tau} H^{({1 \over 4})_{+}}} &  \lesssim \sup_{t' \in [0, \tau]} \tau \| u^K(t_{n} + t') \|_{l^\infty_{\tau}H^{({1 \over 4})_{+}}}
\| \Pi_{\tau^{- {1 \over 2}}} u^K(t_{n}+ t') \|_{l^8_{\tau} W^{({1 \over 4})_{+}, 4 }}^2 \\
& \lesssim  \sup_{t' \in [0, \tau]} \tau^{3 \over 4} \sup_{t' \in [0, \tau]} \tau \| u^K(t_{n} + t') \|_{l^\infty_{\tau}H^{({1 \over 4})_{+}}}
\| \Pi_{\tau^{- {1 \over 2}}} u^K(t_{n}+ t') \|_{l^4_{\tau} W^{({1 \over 4})_{+}, 4 }}^2.
\end{align*}
Hence by using again \eqref{prodd1}, Proposition \ref{propunifdisc} and \eqref{bshift} we finally get that
$$
\|\Tkappa (u^K(t_{n})) \|_{l^4_{\tau} H^{({1 \over 4})_{+}}} \lesssim C_{T} \tau^{3 \over 4}
$$
if $s_{0}>{1 \over 4}.$ We thus deduce \eqref{R2n2}.

It remains to prove \eqref{R2n3}. We now use \eqref{prodd3bisbis} and \eqref{prodd3bisbis2} to get that
\begin{align*}
\| \mathcal{R}_{2, n} \|_{X^{0, b}_{\tau} } &\lesssim \|  \red \mathcal{F}_{2}(t_n) \black\|_{X^{0, -{3\over 8} }_{\tau} } \\
& \lesssim \| u^K(t_{n}) \|_{X^{0, {3 \over 8} }_{\tau}}^2   \|\Tkappa (u^K(t_{n})) \|_{l^\infty_{\tau} H^{({1 \over 2})_{+}}}\\
&\qquad +  \| u^K(t_{n}) \|_{X^{0, {3 \over 8} }_{\tau}}  \|\Tkappa (u^K(t_{n})) \|_{l^4_{\tau} H^{({1 \over 4})_{+}}}
   \| u^K(t_{n}) \|_{l^\infty_{\tau} H^{({1 \over 2})_{+}}} \\
&\qquad +\| u^K(t_{n}) \|_{X^{0, {3 \over 8} }_{\tau}}  \|\Tkappa (u^K(t_{n})) \|_{l^4_{\tau} H^{({1 \over 4})_{+}}}
    \|\Tkappa (u^K(t_{n})) \|_{l^\infty_{\tau} H^{({1 \over 2})_{+}}}\\
&\qquad +   \|\Tkappa (u^K(t_{n})) \|_{l^4_{\tau} H^{({1 \over 4})_{+}}}^2 \|\Tkappa (u^K(t_{n})) \|_{l^\infty_{\tau} H^{({1 \over 2})_{+}}}.
\end{align*}
Note that, for the last term in the above right-hand side, we have used that in the estimate \eqref{prodd3bisbis}, we can replace in the right hand side the norm $\| u_{n}\|_{X_{\tau}^{0, {3 \over 8}}}$ by the norm $ \|u_{n}\|_{l^4_{\tau} H^{( {1 \over 4})_{+}}}$ by using the Sobolev embedding in space instead of the Bourgain estimate~\eqref{prodd1}. Since we have the obvious estimate $ \| u^K(t_{n}) \|_{l^\infty_{\tau} H^{({1 \over 2})_{+}}} \lesssim \| u^K \|_{L^\infty_{T} H^{({1 \over 2})_{+}}}$ and since
$$
\|\Tkappa (u^K(t_{n})) \|_{l^4_{\tau} H^{({1 \over 4})_{+}}} \lesssim T^{1 \over 4}  \|\Tkappa (u^K(t_{n})) \|_{l^\infty_{\tau} H^{({1 \over 2})_{+}}},
$$
it only remains to estimate  $\|\Tkappa (u^K(t_{n})) \|_{l^\infty_{\tau} H^{({1 \over 2})_{+}}}$. From standard product estimates since $H^{({1 \over 2})_{+}}$ is an algebra, we get that
$$
\|\Tkappa (u^K(t_{n})) \|_{l^\infty_{\tau} H^{({1 \over 2})_{+}}} \lesssim \tau \|u^K(t_{n}) \|_{l^\infty_{\tau} H^{({1 \over 2})_{+}}}^3 \lesssim C_{T} \tau.
$$
This concludes the proof.
\end{proof}

%%%%%%%%%%%%%%%%%%%%%%%%%%%%%%%%%%%%%%%%%%%%%%%%%%%%%%
\section{Proof of Theorem \ref{maintheo}} \label{section7}
%%%%%%%%%%%%%%%%%%%%%%%%%%%%%%%%%%%%%%%%%%%%%%%%%%%%%%

We first observe that thanks to \eqref{u-uK}, we have from the triangle inequality that
\begin{equation}
\label{erreurtotale}
\|u(t_{n})- u^n_{\tau}\|_{L^2} \leq  C_{T}\tau^{s_{0} \alpha \over 2} + \|u^K(t_{n}) - u^n_{\tau}\|_{L^2} \leq C_{T}\tau^{s_{0} \alpha \over 2} + \| e^n\|_{l^\infty_{\tau}L^2},
\end{equation}
where $e^{n}$ solves \eqref{erroreq}. To get the error estimates of Theorem \ref{maintheo}, it thus suffices to estimate $ \|e^n\|_{X_{\tau}^{0,b}}$ for some $b\in (1/2, 5/8)$  thanks to \eqref{sobbourg}. Note that there are two parts in the total error, the space discretization part above and the time discretization error on the right-hand side of \eqref{erroreq} which is estimated in Lemma \ref{lemmaR1n} and Lemma \ref{lemmaR2n}. We shall optimize the total error by choosing the best possible $\alpha$ as regularity allows.

We first prove~\eqref{final1}. For very rough data, when $0 < s_{0} \leq 1/4$, we need the estimate \eqref{prodd1} without loss. This forces us to choose $K= \tau^{- { 1 \over 2}}$, hence $\alpha = 1$ without allowing us to optimize the error. We thus obtain from Lemma \ref{lemmaR1n} and Lemma \ref{lemmaR2n} that
\beq \label{erreurpart1}
\| \mathcal{R}_{1,n} \|_{X_{\tau}^{0,b}} + \| \mathcal{R}_{2, n} \|_{X_{\tau}^{0,b}}
\leq C_{T} \left(   \tau^{ s_{0}- { 0_{+}}} +\tau^{ 1 \over 4}\right) \leq C_{T}\tau^{ s_{0}- { 0_{+}}}.
\eeq
Next, we decompose
\allowdisplaybreaks[4]
\beq
\label{DecGn}
G_{n} =L_{n} - Q_{n} + C_{n}
\eeq
with
\begin{multline}
\label{defLfin}
L_{n}= { 1 \over \tau} \left( - 2 i \left[ \mathcal{J}_{1}^r(\Pi_{\tau^{- {1 \over 2}} }  \overline{e}^{n},  \Pi_{K^+}u^K(t_{n}),
\Pi_{\tau^{- {1 \over 2}} }u^K(t_{n})) +  \mathcal{J}_{1}^r(\Pi_{\tau^{- {1 \over 2}} } \overline{u}^K(t_{n}) ,  \Pi_{K^+}e^{n},
\Pi_{\tau^{- {1 \over 2}} }u^K(t_{n})) \right. \right. \\
\left. \left. {}+ \mathcal{J}_{1}^r(\Pi_{\tau^{- {1 \over 2}} } \overline{u}^K(t_{n}) ,  \Pi_{K^+} u^K(t_{n}),
\Pi_{\tau^{- {1 \over 2}} }e^{n})
\right] \right.\\
\left. - i \left[ \mathcal{J}_{2}^r (\Pi_{K}\overline{e}^n, \Pi_{\tau^{- {1 \over 2}}} u^K(t_{n}),   \Pi_{\tau^{- {1 \over 2}}} u^K(t_{n}))   +\mathcal{J}_{2}^r (\Pi_{K} \overline{u}^K(t_{n}), \Pi_{\tau^{- {1 \over 2}}} e^{n},   \Pi_{\tau^{- {1 \over 2}}} u^K(t_{n}))
\right. \right.\\
{}+\left. \left.\mathcal{J}_{2}^r (\Pi_{K} \overline{u}^K(t_{n}), \Pi_{\tau^{- {1 \over 2}}} u_{K}(t_{n}),   \Pi_{\tau^{- {1 \over 2}}} e^{n}) \right]\right),
\end{multline}
\begin{multline}
\label{defQfin}
Q_{n}= { 1 \over \tau} \left(- 2 i \left[    \mathcal{J}_{1}^r(\Pi_{\tau^{- {1 \over 2}} }  \overline{e}^{n},  \Pi_{K^+} e^n,
\Pi_{\tau^{- {1 \over 2}} }u^K(t_{n})) +  \mathcal{J}_{1}^r(\Pi_{\tau^{- {1 \over 2}} }  \overline{e}^{n},  \Pi_{K^+}u^K(t_{n}),
\Pi_{\tau^{- {1 \over 2}} } e^n)  \right. \right. \\
\left. \left. {}+  \mathcal{J}_{1}^r(\Pi_{\tau^{- {1 \over 2}} }  \overline{u}^K(t_{n}) ,  \Pi_{K^+} e^n,
\Pi_{\tau^{- {1 \over 2}} } e^n)
\right] \right.\\
\left. - i \left[ \mathcal{J}_{2}^r (\Pi_{K}\overline{e}^n, \Pi_{\tau^{- {1 \over 2}}} e^n,   \Pi_{\tau^{- {1 \over 2}}} u^K(t_{n}))
+ \mathcal{J}_{2}^r (\Pi_{K}\overline{e}^n, \Pi_{\tau^{- {1 \over 2}}} u^K(t_{n}),   \Pi_{\tau^{- {1 \over 2}}} e^n) \right. \right.\\
\left. \left. {}+ \mathcal{J}_{2}^r (\Pi_{K}\overline{u}^K(t_{n}), \Pi_{\tau^{- {1 \over 2}}} e^n ,   \Pi_{\tau^{- {1 \over 2}}} u^K(t_{n}))\right] \right),
\end{multline}
\allowdisplaybreaks[3]
and
\beq
\label{defTfin}
C_{n}= { 1 \over \tau} \left(- 2 i    \mathcal{J}_{1}^r(\Pi_{\tau^{- {1 \over 2}} }  \overline{e}^{n},  \Pi_{K^+} e^n,
\Pi_{\tau^{- {1 \over 2}} } e^n)  - i \mathcal{J}_{2}^r (\Pi_{K}\overline{e}^n, \Pi_{\tau^{- {1 \over 2}}} e^n,   \Pi_{\tau^{- {1 \over 2}}} e^n) \right).
\eeq
By using Lemma \ref{bourgainfaciled} and \eqref{erreurpart1},  we get from \eqref{erroreq} that
$$
\|e^n \|_{X^{0, b}_{\tau}} \leq   C_T T_{1}^{\varepsilon_{0}} \|G_{n}\|_{X^{0, - { 3 \over 8}}_{\tau}} +  C_{T}\tau^{ s_{0}- { (0)_{+}}},\qquad n\tau\le T_1,
$$
where $\varepsilon_{0}= 5/8 - b >0$. Next, we have that
$$
\|G_{n}\|_{X^{0, - { 3 \over 8}}_{\tau}} \leq   \|L_{n}\|_{X^{0, - { 3 \over 8}}_{\tau}} +  \|Q_{n}\|_{X^{0, - { 3 \over 8}}_{\tau}} + \|C_{n}\|_{X^{0, - { 3 \over 8}}_{\tau}}.
$$
To estimate the right-hand side, we use the equivalent definitions \eqref{relationJ1}, \eqref{J2rel} and again \eqref{prodd3} and \eqref{shiftt} (we recall that for this case we choose $K \tau^{ 1 \over 2}= 1$). This yields
$$
\|e^n \|_{X^{0, b}_{\tau}} \leq    C_{T} T_{1}^{\varepsilon_{0}}  \left(   \|e^n \|_{X^{0, b}_{\tau}}
+  \|e^n \|_{X^{0, b}_{\tau}}^2 +  \|e^n \|_{X^{0, b}_{\tau}}^3 \right) +  C_{T}\tau^{ s_{0}- { ( 0)_{+}}}.
$$
By choosing $T_{1}$ sufficiently small we thus get that
$$
\|e^n \|_{X^{0, b}_{\tau}} \leq  C_{T}\tau^{ s_{0}- { (0)_{+}}}.
$$
This proves the desired estimate \eqref{final1} for $ 0 \leq n \leq N_{1}= T_{1}/ \tau$. We can then iterate the argument on  $ T_{1}/ \tau \leq  n \leq 2 T_{1}/\tau$ and so on to get the final estimate.
We thus finally obtain from \eqref{erreurtotale} that 
$$\|u(t_{n})- u^n_{\tau}\|_{L^2} \leq  C_{T} (\tau^{s_{0}\over 2} + \tau^{s_{0}-(0)_{+}})$$
which means that for every $0 < \varepsilon < s_{0}$, we have for some $C_{T}$ (depending on $\varepsilon$) the estimate
$$\|u(t_{n})- u^n_{\tau}\|_{L^2} \leq  C_{T} (\tau^{s_{0}\over 2} + \tau^{s_{0}-\varepsilon}).$$
Since we can always choose $\varepsilon$ small enough so that $ s_{0}- \varepsilon >s_{0}/2$, we get
\eqref{final1}.

We next prove \eqref{final2}. We follow the same lines, but we can now optimize the total error. From Lemma \ref{lemmaR1n} and Lemma \ref{lemmaR2n}, we get that
$$
\| \mathcal{R}_{1,n} \|_{X_{\tau}^{0,b}} + \| \mathcal{R}_{2, n} \|_{X_{\tau}^{0,b}}
\leq C_{T} \left(   (K \tau^{1 \over 2})^{1 \over 2}   \tau^{ s_{0}- { ( { 1 \over 8})_{+}}} + (K \tau^{1 \over 2})^{1 \over 2} \tau^{5 \over 8}\right)
\leq C_{T} (K \tau^{1 \over 2})^{1 \over 2}   \tau^{ s_{0}- { ( { 1 \over 8})_{+}}}.
$$
We thus choose  $K$ such that $  (K \tau^{1 \over 2})^{1 \over 2}   \tau^{ s_{0}- { ( { 1 \over 8})_{+}}} = { 1 \over K^{s_{0}}}$ which gives
\beq
\label{choicek1}
K= \tau^{-\alpha/2} = \tau^{ - \tfrac{  s_{0} +  ({ 1 \over 8})_{-}}{s_{0}+ {1 \over 2}}},  \quad \alpha  = 2  \tfrac{  s_{0} +  ({ 1 \over 8})_{-}}{s_{0}+ {1 \over 2}} =
2 \left({1 -  \tfrac{1}{2 s_{0}+ 1} \left(\tfrac{3}{4}\right)_{+}}\right).
\eeq
Note that we have $\alpha \in [1, 2]$ since $1/4  < s_{0} \leq 1/2$, and further
\beq
\label{erreurpart2}
\| \mathcal{R}_{1,n} \|_{X_{\tau}^{0,b}} + \| \mathcal{R}_{2, n} \|_{X_{\tau}^{0,b}}
\leq C_{T}\tau^{  s_{0}(  1 -  { 1 \over 2 s_{0}+ 1} ( { 3 \over 4})_{+})}.
\eeq
By using Lemma \ref{bourgainfaciled},   we get from \eqref{erroreq} that
\beq
\label{cestfini1}
\|e^n \|_{X^{0, b}_{\tau}} \leq   C_T T_{1}^{\varepsilon_{0}} (  \|L_{n}\|_{X^{0, - { 3 \over 8}}_{\tau}}
  + \|Q_{n}\|_{X^{0, - { 3 \over 8}}_{\tau}} + \|C_{n}\|_{X^{0, - { 3 \over 8}}_{\tau}})
  +   C_{T}\tau^{  s_{0}(  1 -  { 1 \over 2 s_{0}+ 1} ( { 3 \over 4})_{+})}.
\eeq
To estimate $L_{n}$, we use the product estimates \eqref{prodd3ter}, \eqref{prodd3bis} to get
$$
\|L_{n}\|_{X^{0, - { 3 \over 8}}_{\tau}} \lesssim  \|u^n\|_{X_{\tau}^{ ( { 1 \over 4} )_{+}, {3\over 8}}}^2 \|e^n\|_{l^\infty_{\tau}L^2}
    +  \|e^n \|_{X_{\tau}^{0, {3 \over 8}}}  \|u^n \|_{X_{\tau}^{0, {3 \over 8}}} \|u^n \|_{X_{\tau}^{ ( { 1 \over 4} )_{+}, {3\over 8}}}
$$
and hence by using again Proposition~\ref{propunifdisc} and \eqref{sobbourg}, we obtain that
\beq
\label{Lnfinal}
\|L_{n}\|_{X^{0, - { 3 \over 8}}_{\tau}} \lesssim C_{T} \|e^n\|_{X^{0, b}_{\tau}}.
\eeq
To estimate $C_{n}$, we use  again \eqref{prodd3} and \eqref{shiftt}. This yields
\beq
\label{Cnfinal}
\|C_{n}\|_{X^{0, - { 3 \over 8}}_{\tau}} \leq C_{T} K \tau^{1 \over 2}  \|e^n \|_{X_{\tau}^{0, {3 \over 8}}}^3.
\eeq
To estimate $Q_{n}$, we use \eqref{prodd3ter} and \eqref{pourlafin} and again Proposition~\ref{propunifdisc}. This yields
\beq
\label{Qnfinal}
\| Q_{n}\|_{X^{0, - { 3 \over 8}}_{\tau}} \leq C_{T} (K \tau^{1 \over 2})^{1 \over 2} \|e^n\|_{X^{0,b}_{\tau}}^2
\eeq
since $s_{0}>1/4$. By setting $ Y = \|e^n\|_{X^{0,b}_{\tau}}/\tau^{  s_{0}(  1 -  { 1 \over 2 s_{0}+ 1} ( { 3 \over 4})_{+})}$, we deduce from the above estimates and \eqref{cestfini1} that
$$
Y \leq  C_{T}T_{1}^{\varepsilon_{0}}\left(  Y +  (K \tau^{ 1 \over 2})^{1 \over 2}\tau^{  s_{0}(  1 -  { 1 \over 2 s_{0}+ 1} ( { 3 \over 4})_{+})} Y^2 +  \left((K \tau^{ 1 \over 2})^{1 \over 2}\tau^{  s_{0}(  1 -  { 1 \over 2 s_{0}+ 1} ( { 3 \over 4})_{+})}\right)^2 Y^3 \right) +  C_{T}.
$$
We can then check that with the choice \eqref{choicek1}, for $1/4  <s_{0} \leq 1/2$,   the exponent $\beta$ of
$$
\tau^\beta = (K \tau^{ 1 \over 2})^{1 \over 2}\tau^{  s_{0}(  1 -  { 1 \over 2 s_{0}+ 1} ( { 3 \over 4})_{+})}
$$
is positive. Hence, we can conclude as before to get \eqref{final2}.

It remains to prove \eqref{final3}. From Lemma \ref{lemmaR1n} and Lemma \ref{lemmaR2n}, we now get that
$$
\| \mathcal{R}_{1,n} \|_{X_{\tau}^{0,b}} + \| \mathcal{R}_{2, n} \|_{X_{\tau}^{0,b}} \leq C_{T}   \tau^{ s_{0}-  ( { 1 \over 8})_{+}}.
$$
We thus choose $K$ such that $\tau^{ s_{0}- { ( { 1 \over 8})_{+}}} = { 1 \over K^{s_{0}}}$ in order to optimize the total error. We find
\beq
\label{choiceK3}
K=  \tau^{ { 1 \over s_{0}} ( { 1 \over 8})_{+}  - 1 } , \quad \alpha = 2 -\tfrac{ 1 }{s_{0}} ( \tfrac{ 1 }{4})_{+} .
\eeq
We can use again \eqref{Cnfinal}, \eqref{Qnfinal} and \eqref{Lnfinal} to obtain from \eqref{erroreq}
$$
\|e^n \|_{X^{0, b}_{\tau}} \leq  C_T T_{1}^{\varepsilon_{0}} \left(  \|e^n \|_{X^{0, b}_{\tau}}
+ (K \tau^{1 \over 2})^{1 \over 2} \|e^n \|_{X^{0, b}_{\tau}}^2 + K \tau^{1 \over 2}  \|e^n \|_{X^{0, b}_{\tau}}^3 \right)
+ C_{T}\tau^{ s_{0}-  ( { 1 \over 8})_{+}}.
$$
Again, by setting $Y= \|e^n \|_{X^{0, b}_{\tau}} / \tau^{ s_{0}-  ( { 1 \over 8})_{+}}$, we get that
$$
Y \leq   C_T T_{1}^{\varepsilon_{0}}\left( Y + (K \tau^{1 \over 2})^{1 \over 2}   \tau^{ s_{0}-  ( { 1 \over 8})_{+}} Y^2
+ \left((K \tau^{1 \over 2})^{1 \over 2}   \tau^{ s_{0}-  ( { 1 \over 8})_{+}}\right)^2 Y^3\right) + C_{T}
$$
and we conclude as before, by observing that the exponent of $\tau$ in
$$
(K \tau^{1 \over 2})^{1 \over 2}   \tau^{ s_{0}-  ( { 1 \over 8})_{+}}
$$
is positive with the choice \eqref{choiceK3}.

%%%%%%%%%%%%%%%%%%%%%%%%%%%%%%%%%%%%%%%%%%%%%%%%%%%%%%%%%%%%%%%%%%%%%
\section{Proof of Lemma \ref{prodd}} \label{sectiontechnic}
%%%%%%%%%%%%%%%%%%%%%%%%%%%%%%%%%%%%%%%%%%%%%%%%%%%%%%%%%%%%%%%%%%%%%

We have to prove \eqref{prodd1}. For this purpose, we adapt the proof in \cite{Tao06}  (which is attributed to N.~Tzvetkov). We first observe that
\begin{equation}\label{trick1}
\|\Pi_{K}u_{n}\|_{l^4_{\tau}L^4}^2= \| (\Pi_{K}u_{n})^2 \|_{l^2_{\tau}L^2}.
\end{equation}
By the definition of the $X^{0,b}_\tau$ norm and by setting  $f_{n}= \e^{-i n\tau \partial_{x}^2} \Pi_{K}u_{n}$,  it is equivalent to prove that
$$
\left\| (\e^{in\tau \partial_{x}^2} f_{n})^2 \right\|_{l^2_{\tau}L^2} \lesssim  K \tau^{1 \over 2} \|f_{n}\|_{H^\frac38_{\tau}L^2}^2.
$$
By using the space-time Fourier transform we shall decompose $\widetilde {f_m}(\sigma, k)$ by using a Littlewood--Paley decomposition with respect to $\sigma$. Note that since $\sigma \in [-\pi/\tau, \pi/\tau]$, there is actually a finite number of terms. We write
$$
f_{n} = \sum_{l\ge 0} f_{n,l}
$$
where $\widetilde {f_{m,l}}(\cdot, k) $ is supported in $ 2^{l-1} \leq  \langle  \sigma  \rangle  \leq 2^{l+1}$ for every $k$. By symmetry and the triangle inequality,  it is  sufficient to prove that
$$
\sum_{ p \leq q}   \| \e^{in\tau \partial_{x}^2} f_{n,p}    \,\e^{in\tau \partial_{x}^2} f_{n,q} \|_{l^2_{\tau}L^2} \lesssim K \tau^{1\over 2}   \|f_{n}\|_{H^{3 \over 8}_{\tau}L^2}^2.
$$
We shall actually prove that there exists $\epsilon >0$ (we shall see that we can take $\epsilon = 1/8$) such that for every $p, q$ with $p \leq q$,
\begin{equation}\label{ap2}
\| \e^{in\tau \partial_{x}^2} f_{n,p}   \,\ e^{in\tau \partial_{x}^2} f_{n,q} \|_{l^2_{\tau}L^2}  \lesssim  K \tau^{1 \over 2}\,2^{\epsilon(p-q)} (2^{{3 \over 8} p} \|f_{n,p}\|_{l^2_{\tau}L^2}) (2^{{3 \over 8} q} \|f_{n,q}\|_{l^2_{\tau}L^2} ).
\end{equation}
Once this inequality is proven, the result follows easily. Indeed, let us set $a_{m}=   2^{{3 \over 8} m} \|f_{n,m}\|_{l^2_{\tau}L^2}$,  $b_{m} = 2^{\epsilon m} \mathrm{1}_{m \leq 0}.$ By Parseval, we have that $|a_{m}| \lesssim   \|f_{n,m}\|_{H^{3 \over 8}_{\tau}L^2}$ and that
$$
\| a_{m}\|_{l^2} \lesssim  \|f_{n}\|_{H^{3 \over 8}_{\tau}L^2}.
$$
Moreover, assuming that \eqref{ap2} is proven we obtain that
$$
\sum_{ p \leq q}   \| \e^{i n\tau \partial_{x}^2} f_{n,p}   \,\e^{i n\tau \partial_{x}^2} f_{n,q} \|_{l^2_{\tau}L^2} \lesssim K \tau^{1 \over 2}  \|( b*a)_{m}  a_{m}) \|_{l^1} \lesssim K \tau^{1 \over 2} \| a_m\|_{l^2}^2
$$
from Cauchy--Schwarz and Young's inequality for sequences (observe that $b \in l^1$) which is the desired estimate.

We shall now prove  \eqref{ap2}. From Parseval and by using \eqref{transl}, we have that
$$
\| \e^{i n\tau \partial_{x}^2} f_{n,p}    \,\e^{i n\tau \partial_{x}^2} f_{n,q} \|_{l^2_{\tau}L^2}^2  =  \sum_{k} \int_{\pi\over \tau}^{\pi \over \tau} \left|  \sum_{k_{1}+k_{2}= k} \int_{ \sigma_{1}+ \sigma_{2}= \sigma}  \widetilde{f_{n, p}}(\sigma_{1}-k_{1}^2, k_{1}) \widetilde{f_{n, q}}(\sigma_{2}-k_{2}^2, k_{2}) \dd \sigma_{1} \right|^2 \dd \sigma.
$$
Now let us notice that we have a nontrivial contribution if $\sigma_{1} - k_{1}^2$ is in the support if  $ \widetilde{f_{n,p}}(\cdot, k_{1})$ and $ \sigma_{2} - k_{2}^2$  in the one of $\widetilde{f_{n,q}}(\cdot, k_{2})$. By periodicity in the $\sigma $  variable, this means that there exist $m_{1}$, $m_{2}$ in $\mathbb{Z}$ such that
$$
\left| \sigma _{1} - k_{1}^2 -  {2m_{1} \pi \over \tau} \right| \lesssim 2^p, \quad  \left| \sigma _{2} - k_{2}^2 -  {2m_{2} \pi \over \tau} \right| \lesssim  2^q.
$$
In other words, we have that $\sigma_{1} -  k_{1}^2 \in  E_{p}$, $\sigma_{2} - k_{2}^2 \in  E_{q}$ where $E_{l}= \cup_{ |m| \leq N} [ {2 m  \pi \over \tau} - 2^l,  {2 m \pi \over \tau} + 2^l]$. Note that since the frequencies $k_{1}^2, k_{2}^2$ are smaller than $K^2$, we can take $N \lesssim \tau K^2.$

By using again  Cauchy--Schwarz, we thus  get that
\begin{equation} \label{ap3}
\| \e^{i n\tau \partial_{x}^2} f_{n,p}    \,\e^{i n\tau \partial_{x}^2} f_{n,q} \|_{l^2_{\tau}L^2}^2 \lesssim M_{p,q}  \| \widetilde{ f_{n,p}}\|_{L^2{l^2}}^2   \| \widetilde{f_{n,q}}\|_{L^2{l^2}}^2,
\end{equation}
where
$$
M_{p,q}= \sup_{k, \sigma} \sum_{k_{1}+ k_{2}= k} \int_{ \sigma_{1} + \sigma_{2}= \sigma, \, \sigma_{1}- k_{1}^2 \in E_{p}, \sigma_{2}- k_{2}^2 \in E_{q} } \dd \sigma_{1}.
$$
To estimate $M_{p,q}$, we observe that only $\sigma  \in k_{1}^2+ k_{2}^2 + E_{p}+ E_{q} \subset  k_{1}^2+ k_{2}^2  + 2E_{q}$ gives a nonzero contribution and that the integral is bounded by a constant times $2^p$. Since $k_{1}+ k_{2}=k$, we have
$$ 
k_{1}^2 + k_{2}^2= {1 \over 2} \left( k^2 + (k_{1} - k_{2})^2\right)
$$
and hence
$$
(k_{1} - k_{2})^2 \in 2 \sigma - k^2 - 4 E_{q}.
$$
Therefore,  $k_{1}- k_{2}$ is constrained in intervals of length $\lesssim  2^{q \over 2}$ and there are at most $2 \tau K^2$ intervals. As a consequence, we obtain that
$$
M_{p,q} \lesssim \tau K^2  2^{q \over 2} 2^p= \tau K^2 2^{ p-q \over 4 } 2^{ 3 p \over 4} 2^{3 q \over 4}.
$$
Taking the square root, we thus deduce \eqref{ap2} from \eqref{ap3}. This concludes the proof of \eqref{prodd1}.

%%%%%%%%%%%%%%%%%%%%%%%%%%%%%%%%%%%%%%%%%%%%%%%%%%%%%%%
\section{Numerical experiment}\label{sec:numerical-ex}
%%%%%%%%%%%%%%%%%%%%%%%%%%%%%%%%%%%%%%%%%%%%%%%%%%%%%%%

In this section we illustrate our main result (Theorem \ref{maintheo}) on the $L^2$ error estimate by a numerical experiment. For this purpose, we solve the periodic Schr\"{o}dinger equation \eqref{nlsO} with initial value
\begin{align*}
u(0) = f_{H^1} + \frac{2\sin x}{2-\cos x}
\end{align*}
on the torus. Here $f_{H^1}$ is a randomized $H^1$ function normalised in $L^2$ (see \cite{KOS19} for details on the construction of $f_{H^1}$). We compare our new integrator \eqref{scheme} with the previously introduced single-filtered Fourier based method \cite{OS18,ORS19} and two standard integration schemes for periodic Schr\"{o}dinger equations: a Lie splitting and  exponential integrator method (see, e.g., \cite{CoGa12,Lubich08}). For the latter, we employ a standard Fourier pseudospectral method for the discretization in space and we choose as largest Fourier mode $K = 2^{10}$ (i.e., the spatial mesh size $\Delta x =  0.0061$). On the other hand,  for our twice-filtered Fourier based integrator, we have to use the relation $K = \tau^{-\alpha/2}$ with $\alpha$ given in \eqref{choiceK3}, as we are in case (iii) of Theorem \ref{maintheo} (recall that $s_0 = 1$). This results in $K \approx \tau^{-5/6}$.

We observe from the experiment that the new twice-filtered Fourier integrator is convergent of order one for rough solutions in $H^1$ whereas the standard discretization techniques as well as our previously introduced single-filtered Fourier based method all suffer from order reduction, see Figure \ref{fig}. In particular, the numerically obtained order for the exponential integrator is reduced down to 0.3, whereas the results for the standard Lie splitting scheme are highly irregular. Both integrators are thus unreliable and inefficient for such low regularity initial data. The single-filtered Fourier based integrator shows a more regular error behaviour for the considered example, however, its order is reduced to 3/4. The only method that is able to integrate the considered low regularity problem appropriately is the twice-filtered Fourier based scheme \eqref{scheme} proposed in this paper. \red For  smooth solutions the new twice-filtered Fourier integrator preforms similar as the Lie splitting and the exponential integrator.  More precisely, for initial values at least in $H^2$ all three schemes converge with first-order accuracy  $\tau$; see, e.g., \cite{ESS16} for the analysis of the Lie splitting method  for $H^2$ data.\black

\begin{figure}[h]
\centering
\includegraphics[width=0.471\linewidth]{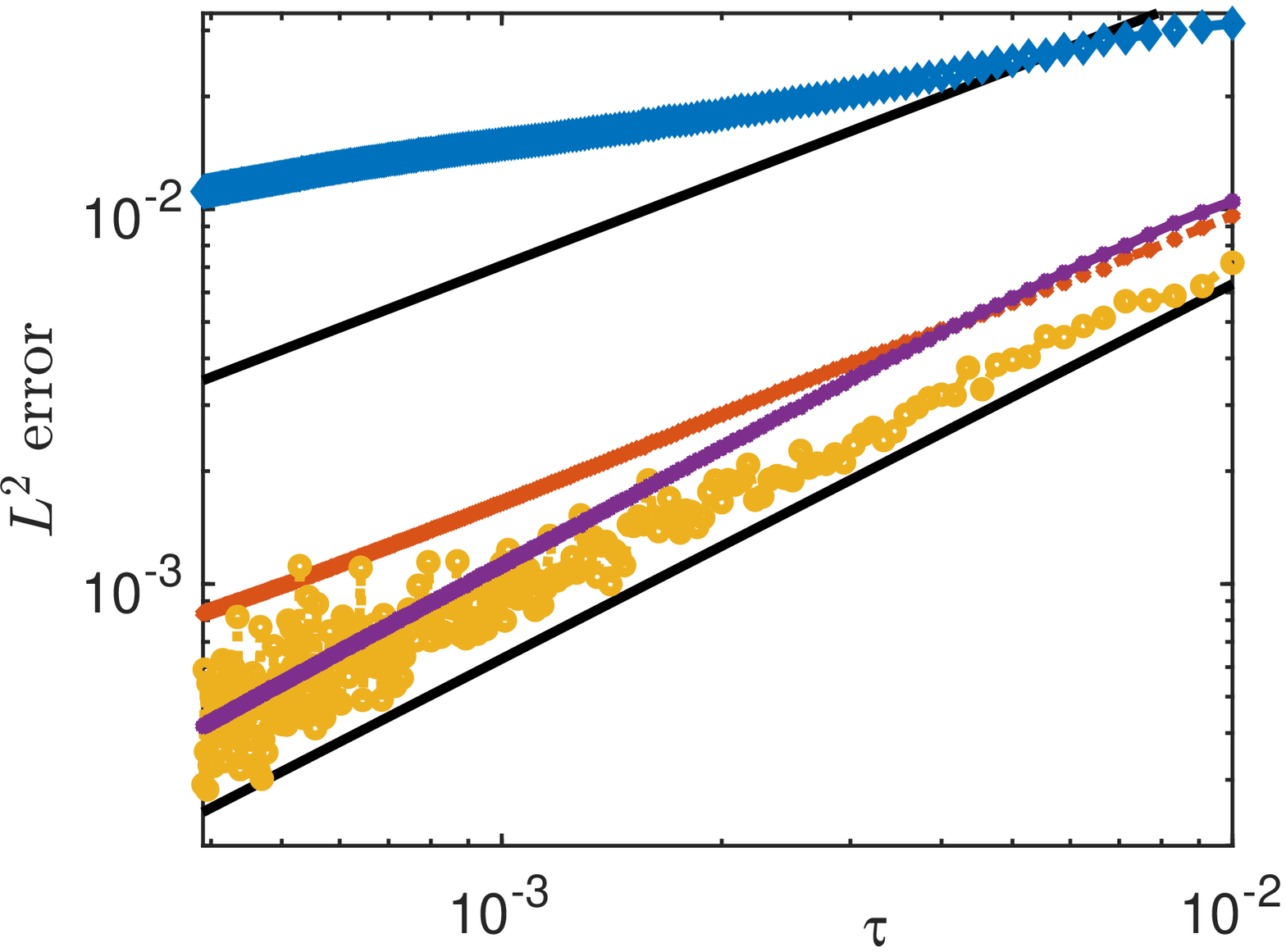}
\hfill
\includegraphics[width=0.48\linewidth]{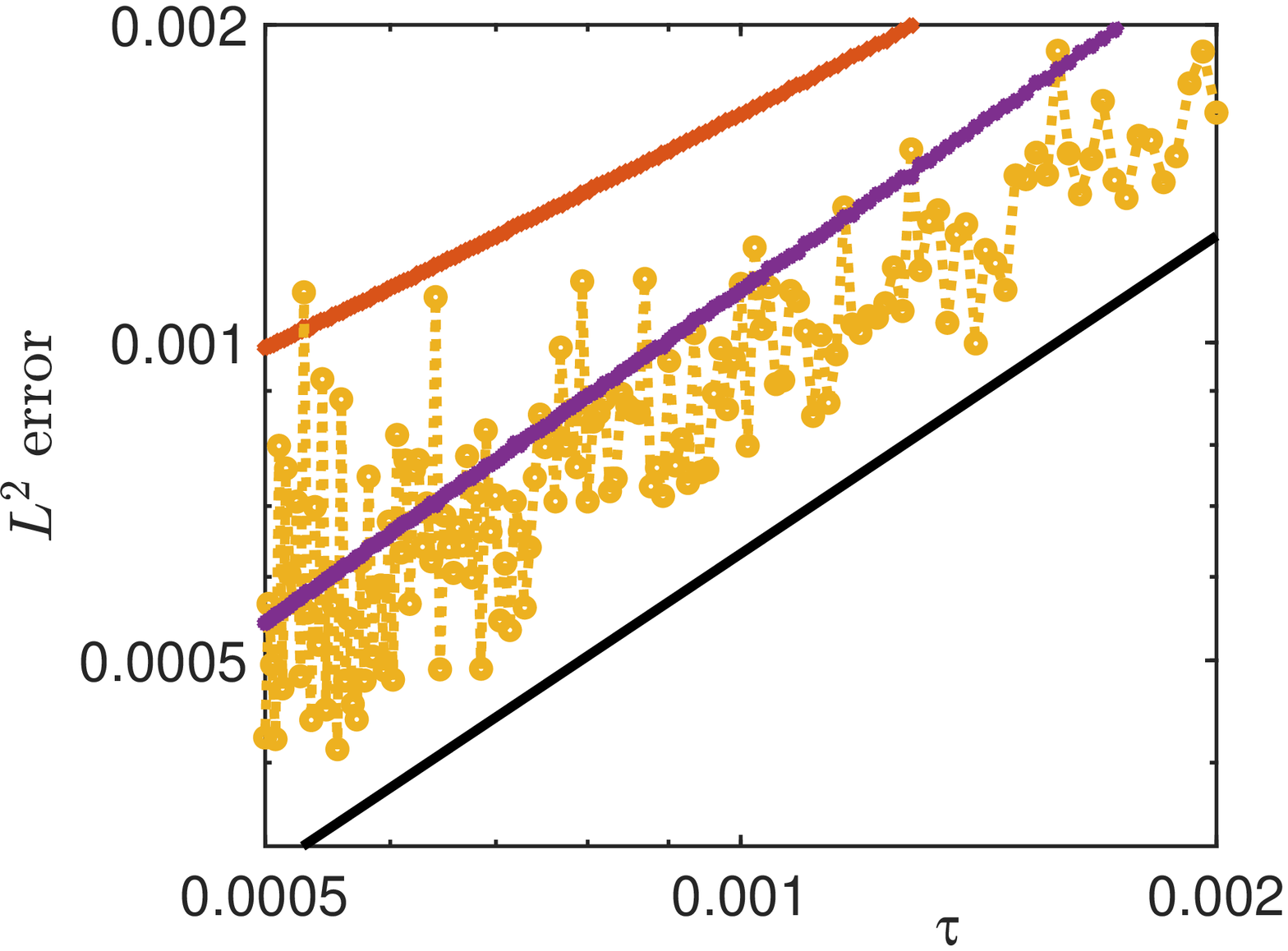}
\caption{$L^2$ error of the new twice-filtered Fourier based scheme \eqref{scheme} (purple), the  Lie splitting scheme (yellow), the exponential integrator (blue), and the original single-filtered Fourier based scheme (red) proposed in \cite{OS18,ORS19}. Left picture: the slope of the (black) reference lines is 1 and $\frac34$, respectively. Right picture: zoom into the region of the left lower corner; the slope of the (black) reference line is 1.}\label{fig}
\end{figure}

\subsection*{Acknowledgements}

{\small
KS has received funding from the European Research Council (ERC) under the European Union’s Horizon 2020 research and innovation programme (grant agreement No. 850941).
}

\end{document}